\documentclass[reqno]{amsart}

\usepackage{amsfonts}
\usepackage{amsmath}
\usepackage{amssymb}
\usepackage{amscd}
\usepackage{mathrsfs}  
\usepackage{amsbsy}
\usepackage{amsthm}
\usepackage{graphicx}
\usepackage{fancyhdr}
\usepackage{epsfig}
\usepackage{color}
\usepackage{xy}
\usepackage{enumitem}

\usepackage{CJKutf8}
\usepackage{inputenc}
\usepackage[encapsulated]{CJK}
\usepackage[CJK, overlap]{ruby}

\usepackage[OT2,OT1]{fontenc}  
\newcommand\cyr{%
\renewcommand\rmdefault{wncyr}%
\renewcommand\sfdefault{wncyss}%
\renewcommand\encodingdefault{OT2}%
\normalfont \selectfont} \DeclareTextFontCommand{\textcyr}{\cyr}

\newcommand{\be}{\begin{equation}}
\newcommand{\ee}{\end{equation}}

\newcommand{\bes}{\begin{equation*}}
\newcommand{\ees}{\end{equation*}}

\newcommand{\inn}[2]{{\langle #1,#2 \rangle}}

\newcommand{\C}{\mathbb{C}}

\newcommand{\bH}{\mathbb{H}}

\newcommand{\N}{\mathbb{N}}

\newcommand{\R}{\mathbb{R}}

\newcommand{\Z}{\mathbb{Z}}

\newcommand{\cK}{\mathcal{K}}

\newcommand{\cO}{\mathcal{O}}

\newcommand{\cS}{\mathcal{S}}

\newcommand{\sinc}{\mathop{\mathrm{sinc}}}
\newcommand{\sF}{\mathscr{F}}
\newcommand{\sD}{\mathscr{D}}
\newcommand{\sP}{\mathscr{P}}

\renewcommand{\Re}{\mathop{\mathrm{Re}}}
\renewcommand{\Im}{\mathop{\mathrm{Im}}}

\newcommand{\Span}{\mathrm{span\,}}
\newcommand{\clos}{\mathrm{clos}}

\newcommand{\wh}[1]{{\widehat{#1}}}
\renewcommand{\rmdefault}{cmr} 
\renewcommand{\sfdefault}{cmr} 
\newtheorem{theorem}{Theorem}
\theoremstyle{plain}

\newtheorem{corollary}{Corollary}

\newtheorem{lemma}{Lemma}

\newtheorem{proposition}{Proposition}
\newtheorem{remark}{Remark}

\newtheorem*{notation*}{Notation}
\numberwithin{equation}{section}

\DeclareMathOperator\Sc{Sc}
\DeclareMathOperator\Ve{Vec}
\DeclareMathOperator\re{Re}
\DeclareMathOperator\im{Im}
\DeclareMathOperator\gr{graph}
\DeclareMathOperator\const{const.}

\newcommand{\mydot}{\,\cdot\,}

\begin{document}

\title{Quaternionic B-Splines}

\author{Jeffrey A. Hogan${}^{1,2}$}
\thanks{${}^1$Corresponding author}
\thanks{${}^2$Research partially supported by ARC grant DP160101537}
\address{School of Mathematical and Physical Sciences, Mathematics Bldg V123, University of Newcastle, University Drive, Callaghan NSW 2308, Australia}
\email{jeff.hogan@newcastle.edu.au}

\author{Peter Massopust${}^3$}
\thanks{${}^3$Research partially supported by DFG grant MA5801/2-1}
\address{Centre of Mathematics, Research Unit M15, Technical University of Munich, Boltzmannstr. 3, 85748 Garching b. Munich, Germany}
\email{massopust@ma.tum.de}

\begin{abstract}
We introduce B-splines on the line of quaternionic order $B_q$ ($q$ in the algebra of quaternions) for the purposes of multi-channel signal and image analysis. The functions $B_q$ are defined first by their Fourier transforms, then as the solutions of distributional differential equation of quaternionic order. The equivalence of these definitions requires properties of quaternionic Gamma functions and binomial expansions, both of which we investigate. The relationship between $B_q$ and a backwards difference operator is shown, leading to a recurrence formula. We show that the collection of integer shifts of  $B_q$ is a Riesz basis for its span, hence generating a multiresolution analysis. Finally, we demonstrate the pointwise and $L^p$ convergence of  the quaternionic B-splines to quarternionic Gaussian functions.
\end{abstract}
%
\dedicatory{This paper is dedicated to the memory of Laureate Professor Jon Borwein, who passed away during its preparation. Jon was a friend and mentor to generations of mathematicians across the globe and has left an incomparable legacy of work spanning multiple disciplines. He was generous with his time and his ideas and was a highly respected and well-loved faculty member at the University of Newcastle in Australia.}

\maketitle
\section{Introduction}\label{sec1}
The extension of the concept of cardinal polynomial B-splines to orders other than $n\in\N$ was first undertaken in \cite{UB,Z}. There, real orders $\alpha > 1$ were considered and in \cite{UB} these splines were named {\em fractional B-splines}. In \cite{FBU} a more general class of cardinal B-splines of complex order or, for short, {\em complex B-splines}, $B_z: \R \to\C$ were defined in the Fourier domain by
\begin{equation}
\sF (B_z)(\omega ) =:  \widehat{B_z} (\omega) := \int_{\R} B_z(t)e^{-i\omega t}\, dt := \left( \frac{1-e^{-i\omega}}{i\omega}\right)^z,
\label{eq Definition Fourier B-Spline}
\end{equation}
for $z\in \C$ with $\Re z  >1$. 

The motivation behind the definition of complex B-splines is twofold. Firstly, there is the need for a continuous family (with respect to smoothness) of analyzing basis functions to close the gap between the the integer-valued smoothness spaces $C^n$ associated with the classical Schoenberg polynomial B-splines. Secondly, there are requirements for a single-band frequency analysis. For some applications, e.g., for phase retrieval tasks, complex-valued analysis bases are needed since real-valued bases can only provide a symmetric spectrum. Complex B-splines combine the advantages of spline approximation with an approximate one-sided frequency analysis. In fact, the spectrum $|\widehat B_{z}(\omega)| $ has the form
\[
|\widehat{B_z}(\omega)| = |\widehat{B_{\Re z}}(\omega)| e^{-i \Im z \ln |\Omega (\omega)|} e^{\Im z \arg \Omega(\omega)},
\]
where $\Omega(\omega) := \frac{1-e^{-i\omega}}{i \omega}$.
Thus, the spectrum consists of the spectrum of a real-valued B-spline, combined with a modulating and a damping factor:
The presence of the imaginary part $\Im z$ causes the frequency components on the negative and positive real axis to be  enhanced with different signs. This has the effect of shifting the frequency spectrum towards the negative or positive frequency side, depending on the sign of $\Im z$. The corresponding bases can be interpreted as approximate single-band filters \cite{FBU}.

For certain types of applications such as geophysical data processing a multi-channel description is required. For instance, seismic data has four channels, each associated with a different kind of seismic wave: the so-called P (Compression), S (Shear), L (Love) and R (Rayleigh) waves. Similarly, the colour value of a pixel in a colour image is composed of three components -- the red, green and blue channels. In order to perform the tasks of processing multi-channel signals and data an appropriate set of analyzing basis functions is required. These basis functions should have the same analytic properties of complex B-splines but should in addition be able to describe  multi-channel structures. In \cite{O1,O2} a set of analyzing functions based on wavelets and Clifford-analytic methodologies were introduced in an effort to process four channel seismic data. A multiresolution structure for the construction of wavelets on the plane for the analysis of four-channel signals was outlined in \cite{HM}. Here we first investigate the mathematical foundations for the application of quaternion-valued basis functions to the analysis of multichannel signals and images.

This paper is organized as follows. In section 2 we outline the properties of the algebra of quaternions that will be required in later sections and set the notation. In section 3 we introduce the quaternionic B-splines $B_q$ via their Fourier transforms, show that in general the semigroup property $B_z*B_w=B_{z+w}$ enjoyed by the complex B-splines, fails when the order is quaternionic, give estimates of the $L^2$ and $L^1$ norms of he quaternionic B-splines, and outline their decay, smoothness, and approximation order. Section 4 concentrates on quaternionic binomial expansions and quaternionic Gamma functions, and results are proved which will allow us to give a  description of $B_q$ in the time domain. This description is delivered in section 5, 
where we also prove a recurrence relation for quaternionic splines, achieved through the consideration of an appropriate backwards difference operator. In section 6 we outline miscellaneous properties of quaternion B-splines and Gamma functions, including rotation-covariance between the quaternionic order and the range of these functions. The refinability and multiscale structure of $B_q$ is the topic of section 7. It is shown that, like the standard B-splines, the quaternionic B-splines are scaling functions in the sense of wavelet theory and their shifts form a Riesz basis for their linear span. The approximation order of the associated projection operators is also investigated. Finally in section 8, the pointwise and $L^p$ convergence properties of the the B-splines to quaternionic Gaussian functions in proved.



%

\section{Notation and Preliminaries}
The real, associative algebra of quaternions ${\mathbb H}={\mathbb H}_{\mathbb R}$ is given by
$${\mathbb H}_{\mathbb R}=\left\{a+\sum_{i=1}^3v_ie_i : a, v_1, v_2, v_3\in{\mathbb R}\right\},$$
where the imaginary units $e_1, e_2, e_3$ satisfy $e_1^2=e_2^2=e_3^2=-1$, $e_1e_2=e_3$, $e_2e_3=e_1$ and $e_3e_1=e_2$. Because of these  relations, ${\mathbb H}$ is a non-commutative algebra. 

Each quaternion $q=a+\sum\limits_{i=1}^3v_ie_i$ may be decomposed as $q=\Sc (q)+\Ve (q)$ where $\Sc (q)=a$ is the {\it scalar part} of $q$ and $\Ve (q)=v=\sum\limits_{i=1}^3v_ie_i$ is the {\it vector part} of $q$. The {\it conjugate} $\overline{q}$ of the real quaternion $q=a+v$ is the quaternion $\overline{q}=a-v$. Note that $q\overline{q}=\overline{q}q=|q|^2=a^2+|v|^2=a^2+\sum\limits_{i=1}^3v_i^2$. 
Note also that if $v=\sum\limits_{j=1}^3v_je_j$ and $w=\sum\limits_{j=1}^3w_je_j$ are quaternionic vectors, then 
\begin{equation}
vw=-\langle v,w\rangle +v\wedge w\label{vec mult},
\end{equation}
where $\langle v,w\rangle =\sum\limits_{j=1}^3v_jw_j$ is the scalar product of $v$ and $w$ and 
$$v\wedge w=(v_2w_3-v_3w_2)e_1+(v_3w_1-v_1w_3)e_2+(v_1w_2-v_2w_1)e_3$$ 
is the vector (cross) product of $v$ and $w$.

If $q=q_0+\sum\limits_{i=1}^3q_ie_i\in{\mathbb H}_{\mathbb C}=\left\{a+\sum\limits_{i=1}^3v_ie_i : a, v_1, v_2, v_3\in{\mathbb C}\right\}$, we define the conjugate $\overline{q}$ of $q$ by $\overline q=\overline{q_0}-\sum\limits_{i=1}^3\overline{q_i}e_i$, where $\overline{q_i}$ is the complex conjugate of the complex number $q_i$. If $p=p_0+\sum\limits_{i=1}^3p_ie_i\in{\mathbb H}_{\mathbb C}$, we define the inner product $\langle p,q\rangle$ to be the complex number  $\langle p,q\rangle =\sum\limits_{i=0}^3p_1\overline q_i=\Sc (p\overline{q})$. We also define $e^q$ by the usual series: $e^q=\sum\limits_{j=0}^\infty \dfrac{q^j}{j!}$. We require the following bounds on $|e^q|$, which we state without proof.
\begin{lemma}\label{exp bounds} Let $q=a+v\in{\mathbb H}_{\mathbb R}$. Then we have
\begin{enumerate}
\item[1.] $|e^q|=e^a\leq e^{|q|}$.
\item[2.] I$f z\in{\mathbb C}$ and $q'=zq\in{\mathbb H}_{\mathbb C}$, then $|e^{q'}|\leq e^{\sqrt 2|q'|}$.
\end{enumerate}
\end{lemma}

If ${\mathbb F}={\mathbb R},\ {\mathbb C},\ {\mathbb H}_{\mathbb R}$ or ${\mathbb H}_{\mathbb C}$, and $1\leq p<\infty$, then $L^p({\mathbb R},{\mathbb F})$ is the Banach space of measurable functions $f:{\mathbb R}\to{\mathbb F}$ for which $\int\limits_{-\infty}^\infty |f(x)|^p\, dx<\infty$, where the meaning of $|f(x)|$ is dependent on ${\mathbb F}$. $L^\infty ({\mathbb R},{\mathbb F})$ is defined similarly.  On $L^2({\mathbb R},{\mathbb C})$, we define an inner product by 
$\langle f,g\rangle =\int\limits_{-\infty}^\infty f(x)\overline{g(x)}\, dx.$
 On $L^2({\mathbb R},{\mathbb H}_{\mathbb C})$, the inner product  is given by
\begin{equation}
\langle f,g\rangle =\Sc\bigg(\int_{-\infty}^\infty f(x)\overline{g(x)}\, dx\bigg)=\int_{-\infty}^\infty\langle f(x),g(x)\rangle\, dx.\label{IP complex quat}
\end{equation}

The  Fourier-Plancherel transform $\sF$ is defined on  $L^1 (\R, {\mathbb F})$ by 
\begin{gather*}
(\sF f) (\xi) := \wh{f} (\xi) := \int_\R f(x) e^{-i \xi x} dx
\end{gather*}
and may be extended to $L^2({\R},{\mathbb F})$, on which it becomes a multiple of a unitary mapping: $\langle\sF f,\sF g\rangle =2\pi \langle f,g\rangle$. 

\section{Quaternionic B-Splines}
For $q=a+v=a+\sum\limits_{j=1}^3v_je_j\in{\mathbb H}_{\mathbb R}$ and
$z\in \C$, we define the quaternionic power $z^q\in{\mathbb H}_{\mathbb C}$ by
\be\label{eq2.1}
z^q := z^a[\cos (|v|\log z)+\frac{v}{|v|}\sin (|v|\log z)].
\ee
This definition of $z^q$ allows for the usual differentiation and integration rules:
\begin{equation}
\frac{d}{dz} z^q = q z^{q-1}\qquad\text{and}\qquad \int z^q dz = \frac{z^{q+1}}{q+1} + \const, \quad q\neq -1.\label{diff/int formulae}
\end{equation}
In general, however, the semigroup property $z^{q_1}z^{q_2}=z^{q_1+q_2}$ fails to hold. The following proposition characterizes the situation in which the semigroup property holds:

\begin{proposition}\label{prop1}
Suppose that $q_1=a_1+v_1, q_2=a_2+v_2\in{\mathbb H}_{\mathbb R}$ is fixed. 
Then  $z^{q_1}z^{q_2}=z^{q_3}$ for all $z$ in a neighborhood of $z=1$ if and only if the set $\{v_1,v_2\}$ is linearly dependent in ${\mathbb R}^3$ and $q_3=q_1+q_2$.
\end{proposition}
\begin{proof} 
First note that with an application of the first of the equations in (\ref{diff/int formulae}), differentiating the equation $z^{q_1}z^{q_2}=z^{q_3}$ with respect to $z$ yields
$$\frac{q_1}{z}z^{q_1}z^{q_2}+z^{q_1}\frac{q_2}{z}z^{q_2}=\frac{q_3}{z}z^{q_3}.$$
Setting $z=1$ in this equation gives $q_1+q_2=q_3$. Furthermore,
\begin{align}
z^{q_1+q_2}&=a^{a_1+a_2}[\cos (|v_1+v_2|\log z)+\frac{v_1+v_2}{|v_1+v_2|}\sin (|v_1+v_2|\log z)]\notag\\
&=z^{q_1}z^{q_2}\notag\\
&=z^{a_1+a_2}[(\cos (|v_1|\log z)+\frac{v_1}{|v_1|}\sin (\log z))(\cos (|v_2|\log z)+\frac{v_2}{|v_2|}\sin (\log z))]\notag\\
&=z^{a_1+a_2}[\cos(|v_1|\log z)\cos (|v_2|\log z)+\frac{v_1v_2}{|v_1||v_2|}\sin (|v_1|\log z)\sin (|v_2|\log z)\notag\\
&\qquad +\frac{v_1}{|v_1|}\sin (|v_1|\log z)\cos (|v_2|\log z)+\frac{v_2}{|v_2|}\sin (|v_2|\log z)\cos (|v_1|\log z)].\label{eq1a}
\end{align}
We now divide both sides of (\ref{eq1a}) by $z^{a_1+a_2}$ and equate the vector parts of both sides of the resulting equation:
\begin{align}
\frac{v_1+v_2}{|v_1+v_2|}\sin (|v_1+v_2|\log z)&=\frac{v_1}{|v_1|}\sin (|v_1|\log z)\cos (|v_2|\log z)\notag\\
&+\frac{v_2}{|v_2|}\sin (|v_2|\log z)\cos (|v_1|\log z)\notag\\
&+\frac{v_1\wedge v_2}{|v_1||v_2|}\sin (|v_1|\log z)\sin (|v_2|\log z).\label{wedge eqn}
\end{align}
Writing $\cos\theta =1-\frac{\theta^2}{2} + \cdots$ and $\sin\theta =\theta - \frac{\theta^3}{6}+ \cdots$, we expand both sides of equation (\ref{wedge eqn}), multiply the  power series on the right hand side  together and extract the coefficient of $(\log z)^2$ on both sides. This yields
$v_1\wedge v_2=0$, which gives the result.
\end{proof}

For $z\in \R$, we also define
\[
z^q_+ := \begin{cases} z^q, & z > 0;\\ 0, & \text{otherwise}.
\end{cases}
\]

We define the B-spline $B_q$ of quaternionic order $q$ (for short {\it quaternionic B-spline})  in the Fourier domain to be the function $\wh{B_q}:\R\to\bH_{\mathbb C}$ given by 
\begin{equation}
\wh{B_q} (\xi) := \left(\frac{1-e^{-i \xi}}{i \xi}\right)^q, \quad \Sc (q) > 1.\label{qB}
\end{equation}
Setting $\Xi (\xi) := \dfrac{1-e^{-i \xi}}{i \xi}$, 
we obtain, via \eqref{eq2.1}, the precise meaning of (\ref{qB}), namely,
\be\label{Bhat}
\wh{B_q} (\xi) = \Xi (\xi)^{\Sc q} \left(\cos (|v|\log \Xi (\xi))+\frac{v}{|v|}\sin (|v|\log \Xi (\xi))\right).
\ee

\begin{figure}[h!]
\begin{center}
\includegraphics[width=5cm, height= 4cm]{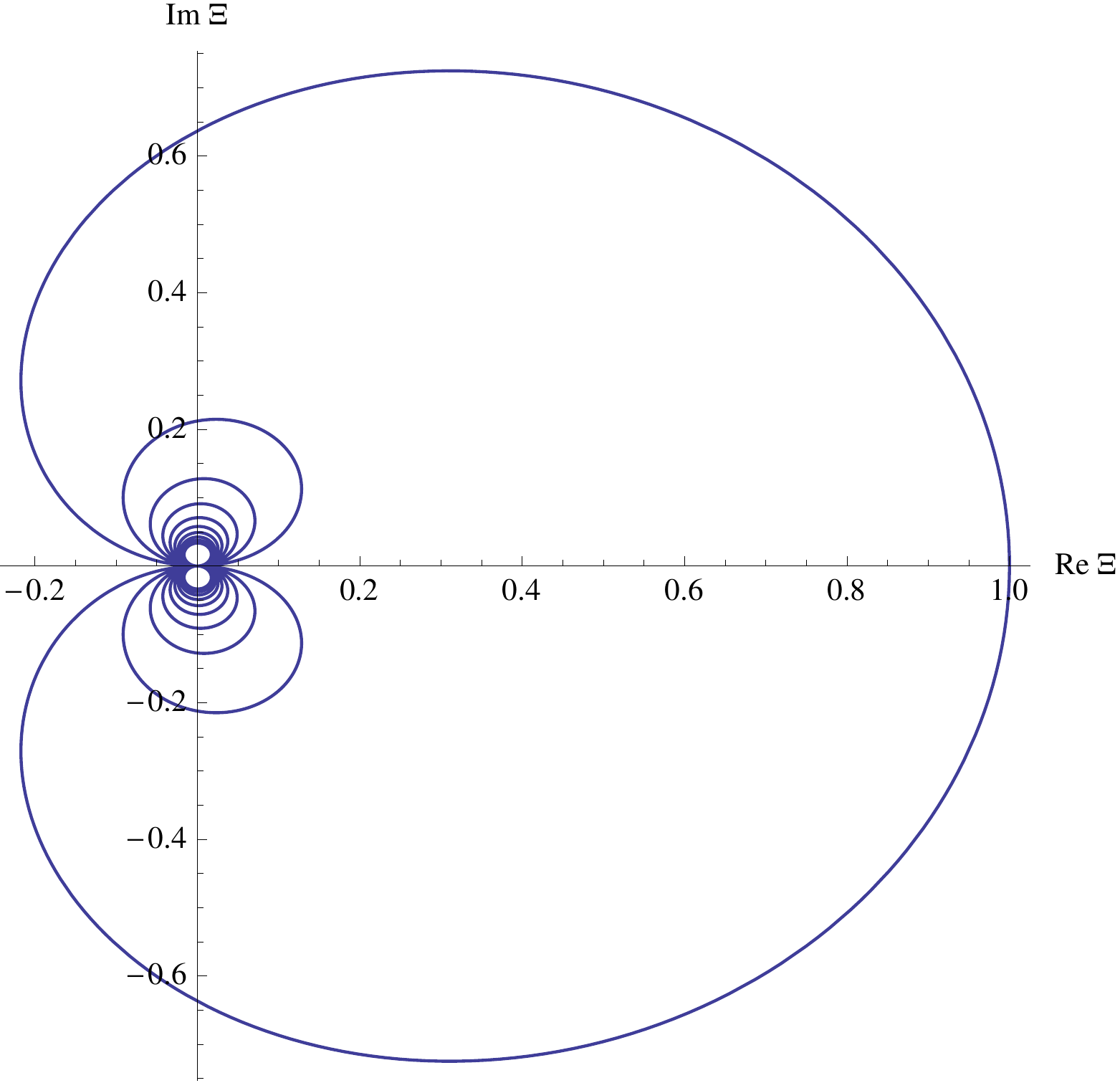}
\caption{The graph of the function $\Xi$ in the complex plane.}\label{fig0}
\end{center}
\end{figure}

The function $\Xi$ has a removable singularity at $\xi = 0$ with $\Xi (0) = 1$.  It follows from $\re \Xi (\xi) = \xi^{-1}\,\sin \xi$ and $\im \Xi (\xi) = \xi^{-1}\,(1-\cos\xi)$, that $\gr\Xi\cap (\R^-\times \{0\}) = \emptyset$. Hence, $\Xi$ and therefore $\wh{B_q}$ are well-defined when $-\pi < \arg\Xi \leq \pi$. Equation \eqref{Bhat} also implies that $\wh{B_q} \in L^2(\R,\bH_{\mathbb C})$ for a fixed $q$ with $\Sc q > \frac12$ and in $L^1(\R,\bH_{\mathbb C})$ for a fixed $q$ with $\Sc (q) > 1$ as the fractional B-splines \cite{UB} satisfy these conditions. For, if $\Sc (q) > \frac12$ then since $|\cos z|^2+|\sin z|^2=\cosh (2\arg z)$, we have
\begin{align}\label{eq2.3a}
\| \wh{B_q}\|_2^2 &= \int_\R |\Xi (\xi)^{ \Sc q}|^2 |\cos (|v|\log \Xi (\xi))+\frac{v}{|v|}\sin (|v|\log \Xi (\xi))|^2 d\xi\nonumber\\
& \leq  \int_\R |\Xi (\xi)^{\Sc q}|^2 (|\cos (|v|\log \Xi (\xi))|^2 + |\sin (|v|\log \Xi (\xi))|^2) d\xi\nonumber\\
& = \int_\R |\Xi (\xi)^{\Sc q}|^2 \cosh (|v|\arg\Xi (\xi)) d\xi\\
& \leq \cosh (\pi |v|)\, \| \wh{B}_{\Sc q}\|_2^2 < \infty.\nonumber
\end{align}
Similarly, if $\Sc (q) > 1$ then
\begin{align*}
\| \wh{B_q}\|_1 &= \int_\R |\Xi (\xi)^{ \Sc q}| |\cos (|v|\log \Xi (\xi))+\frac{v}{|v|}\sin (|v|\log \Xi (\xi))| d\xi\\
& \leq \int_\R |\Xi (\xi)^{\Sc q}|\, \sqrt{\cosh (|v|\arg\Xi (\xi ))} d\xi\\
& \leq \sqrt{\cosh (\pi |v|)}\, \| \wh{B}_{\Sc q}\|_1 < \infty.
\end{align*}
Note that $\wh{B_q}\in L^1(\R,\bH_{\mathbb C})$ implies that $B_q$ is uniformly continuous on $\R$.

The next results follow directly from the corresponding properties of fractional B-splines \cite{UB}. For a real number $s\geq 0$ and $1\leq p\le \infty$,  the Bessel potential space $H^{s,p}({\mathbb R},{\mathbb H})$ is given by
$$H^{s,p}({\mathbb R},{\mathbb H}_{\mathbb C})=\left\{f\in L^p({\mathbb R},{\mathbb H}_{\mathbb C}) : \sF^{-1}[(1+|\xi |^2)^{s/2}\sF f]\in L^p({\mathbb R},{\mathbb H}_{\mathbb C})\right\}.$$
\begin{proposition} 
Let $B_q$ be a quaternionic B-splines with $\Sc (q) >\frac12$. Then $B_q$ enjoys the following properties:
\begin{enumerate}
\item[\emph{(i)}] \emph{Decay}: $B_q \in \cO(|\xi|^{-(\Sc (q))})$ as $|\xi|\to \infty$.
\item[\emph{(ii)}] \emph{Smoothness}: $B_q\in H^{s,p} (\R,\bH_{\mathbb C})$ for $1\leq p\leq\infty$ and $0\leq s < \Sc (q) + \frac1p$.
\item[\emph{(iii)}] \emph{Reproduction of Polynomials}: $B_q$ reproduces polynomials up to order $\lceil\Sc (q)\rceil$, where the ceiling function $\lceil\mydot\rceil:\R\to \Z$ is given by $r\mapsto \min\{n\in \Z : n\geq r\}$.
\end{enumerate}
\end{proposition}
\section{Quaternionic Binomial Expansions}

Given $q=a+v\in{\mathbb H}_{\mathbb R}$, $a := \Sc (q)$, and $j\in \N$, we define the \emph{quaternionic Pochhammer symbol} by
\[
(q)_j := q(q-1)\cdots (q-j+1),
\]
and the \emph{quaternionic binomial coefficient} by
\[
\binom{q}{j} := \dfrac{(q)_j}{j!}. 
\]
We also introduce the quaternionic Gamma function $\Gamma$ by setting
\be\label{gamma}
\Gamma (q) := \int_0^\infty t^{a-1}\cos (|v|\log t) e^{-t} dt + \frac{v}{|v|}\,\int_0^\infty t^{a-1}\sin (|v|\log t)] e^{-t} dt.
\ee 
Note that if $q\in{\mathbb H}_{\mathbb R}$, then $\Gamma (q)\in{\mathbb H}_{\mathbb R}$. The integrals on the right-hand side converge since
\begin{align*}
|\Gamma (q)|^2 & = \Gamma (q)\overline{\Gamma (q)}\\ 
& = \left( \int_0^\infty t^{a-1}\cos (|v|\log t) e^{-t} dt\right)^2 + \left( \int_0^\infty t^{a-1}\sin (|v|\log t) e^{-t} dt\right)^2\\
& \leq \left( \int_0^\infty t^{a-1}e^{-t} dt\right)^2 + \left( \int_0^\infty t^{a-1} e^{-t} dt\right)^2 = 2\, |\Gamma (a)|^2.
\end{align*}
The integrals 
\[
\int_0^\infty t^{a-1}\cos (|v|\log t) e^{-t} dt\quad\text{and}\quad\int_0^\infty t^{a-1}\sin (|v|\log t)] e^{-t} dt
\] 
can be explicitly computed yielding
\[
\tfrac12(\Gamma (a - i |v|) + \Gamma (a+i|v|)\quad\text{and}\quad \tfrac{i}{2}(\Gamma (a - i |v|) - \Gamma (a+i|v|),
\]
respectively. Hence, the quaternionic Gamma function \eqref{gamma} can be extended to include values $a\in\R_-\setminus \Z_-$.

Using \eqref{eq2.1}, we write \eqref{gamma} in the more succinct notation
\be\label{eq2.3}
\Gamma (q) := \int_0^\infty t^{q-1} e^{-t} dt.
\ee
Integration by parts in the integral  \eqref{eq2.3}, where we use the differentiation formula (\ref{diff/int formulae}), 
produces the functional equation
\be\label{recur}
\Gamma (q+1) = q \, \Gamma (q), \qquad \Sc (q) > 0.
\ee
For the quaternionic Pochhammer symbol and the quaternionic binomial coefficient we thus have
\begin{align*}
(q)_j = \frac{\Gamma (q+1)}{\Gamma (q-j+1)}\qquad\text{and}\qquad \binom{q}{j} = \frac{\Gamma(q+1)}{\Gamma (q-j+1) \Gamma (j)}.
\end{align*}
Below, we require an asymptotic estimate of the quaternionic Gamma function. For this purpose we first derive the following Gau{\ss}-type limit representation of $\Gamma (q)$.

\begin{lemma}
Let $q = a + v\in \bH_{\mathbb R}$ with $a > 0$. Then
\be\label{gauss}
\Gamma (q) = \lim_{n\to\infty} \frac{1\cdot 2 \cdots n}{q(q+1)\cdots (q+n)}\,n^q.
\ee
\end{lemma}
\begin{proof}
The proof employs the ideas outlined in \cite[10.1]{arfken}. Set
\begin{equation}
F(q,n) := \int_0^\infty \left(1 - \frac{t}{n}\right)^n\,\chi_{[0,n]} (t)\, t^{q-1}\,dt,\label{F defn}
\end{equation}
where $\chi$ denotes the characteristic function. The dominated convergence theorem implies that
\[
\lim_{n\to\infty} F(q,n) = \int_0^\infty e^{-t}\, t^{q-1}\,dt = \Gamma (q).
\]
Changing variables to $u := \dfrac{t}{n}$ in (\ref{F defn}) gives $F(q,n) = n^q \,\int\limits_0^1 (1-u)^n u^{q-1} \,du$ and
Integrating by parts yields
\begin{align*}
F(q,n) & = n^q \left[(1-u)^n q^{-1} u^q \Big\vert_0^1 + \frac{n}{q}\,\int_0^1  (1-u)^n u^{q} \,du\right]\\
& = n^q\,\frac{n}{q}\,\int_0^1  (1-u)^n u^{q} \,du.
\end{align*}
Repeating $(n-1)$--times gives
\begin{align*}
F(q,n) & = n^q\,\frac{n(n-1)\cdots 1}{q(q+1)\cdots (q+n-1)}\,\int_0^1  u^{q+n-1} \,du\\
& = n^q\,\frac{n(n-1)\cdots 1}{q(q+1)\cdots (q+n)},
\end{align*}
which implies \eqref{gauss}.
\end{proof}

From \eqref{recur}, one obtains the following asymptotic behavior of $\Gamma (q)$: if $n\in \N$,
\begin{align}
\frac{\Gamma (q + n)}{\Gamma (q)} = q (q+1) \cdots (q+n-1) = q^n (1 + \cO (q^{-1})), \quad \Sc (q) > 0.\label{Gamma quotient}
\end{align}

We have the following quaternionic binomial expansion.

\begin{theorem}\label{thm0}
Let $q=a+v\in{\mathbb H}_{\mathbb R}$ with $a>0$ and $z\in {\mathbb C}$ with $|z| \leq 1$. Then 
\be\label{qbin}
(1+z)^q=\sum_{j=0}^\infty\binom{q}{j}z^j.
\ee
\end{theorem}

\begin{proof} Let $f(z)=(1+z)^q=(1+z)^a[\cos (|v|\log (1+z))+\dfrac{v}{|v|}\sin (|v|\log (1+z))]$ and $f_0(z)=(1+z)^a\cos (|v|\log (1+z))$, $f_1(z)=(1+z)^a\sin(|v|\log (1+z))$ so that $f(z)=f_0(z)+\dfrac{v}{|v|}f_1(z)$. Since $f_0$ is analytic on the open unit ball $\{z\in C;\, |z|<1\}$, we have
\[
f_0(z)=\sum_{j=0}^\infty\frac{f_0^{(j)}(0)}{j!}z^j,\qquad\qquad (|z|<1).
\]
Note that 
$$f_0'(z)=(1+z)^{a-1}[a\cos (|v|\log (1+z))-|v|\sin (|v|\log (1+z))].$$
In fact, there are constants $\alpha_j$ and $\beta_j$ such that 
\begin{equation}
f_0^{(j)}(z)=(1+z)^{a-j}[\alpha_j\cos (|v|\log (1+z))+\beta_j\sin (|v|\log (1+z))].\label{jth deriv}
\end{equation}
Let $\gamma_j=\left(\begin{matrix}\alpha_j\\\beta_j\end{matrix}\right)$ and observe that $\gamma_0=\left(\begin{matrix}1\\0\end{matrix}\right)$. Differentiating both sides of (\ref{jth deriv}) with respect to $z$ gives
\begin{align*}
f_0^{(j+1)}(z) & =(1+z)^{a-j-1}[((a-j)\alpha_j+|v|\beta_j)\cos(|v|\log (1+z))\\
& \qquad +(-|v|\alpha_j+(a-j)\beta_j)\sin (|v|\log (1+z))]
\end{align*}
so that 
$$\left(\begin{matrix}\alpha_{j+1}\\\beta_{j+1}\end{matrix}\right)=\left(\begin{matrix}a-j&|v|\\-|v|&a-j\end{matrix}\right)\left(\begin{matrix}\alpha_j\\\beta_j\end{matrix}\right),$$
i.e., $\gamma_{j+1}=(A-jI)\gamma_j$ where $A=\left(\begin{matrix}a&|v|\\-|v|&a\end{matrix}\right)$. Hence
$$\gamma_j=(A-(j-1)I)\gamma_{j-1}=(A-(j-1)I)(A-(j-2)I)\gamma_{j-2}=\cdots =\prod_{\ell =0}^{j-1}(A-\ell I)\gamma_0.$$
The matrix $A$ may be diagonalized over ${\mathbb C}$ or ${\mathbb H}_\R$. We note that $A$ has quaternionic eigenvalues $q=a+v$ with corresponding eigenvector $\left(\begin{matrix}-v/|v|\\1\end{matrix}\right)$ and $\bar q=a-v$ with corresponding eigenvector $\left(\begin{matrix}v/|v|\\1\end{matrix}\right)$. Note that the eigendecomposition of $A$ is not unique. Let $P=\left(\begin{matrix}-v/|v|&v/|v|\\1&1\end{matrix}\right)$. Then $P^{-1}=\dfrac{1}{2}\left(\begin{matrix}v/|v|&1\\-v/|v|&1\end{matrix}\right)$ and
$P^{-1}AP=D=\left(\begin{matrix}q&0\\0&\bar q\end{matrix}\right)$.
Hence,
\begin{align*}
\prod_{\ell =0}^{j-1}(A-\ell I)&=\prod_{\ell =0}^{j-1}(PDP^{-1}-\ell I)\\
&=P\prod_{\ell =0}^{j-1}(D-\ell I)P^{-1}\\
&=P\prod_{\ell =0}^{j-1}\left(\begin{matrix}q-\ell&0\\0&\bar q-\ell\end{matrix}\right)P^{-1}\\
&=P\left(\begin{matrix}\prod_{\ell =0}^{j-1}(q-\ell )&0\\0&\prod_{\ell =0}^{j-1}(\bar q-\ell )\end{matrix}\right)P^{-1}\\
&=P\left(\begin{matrix}q_j&0\\0&\bar q_j\end{matrix}\right)P^{-1}=\frac{1}{2}\left(\begin{matrix}q_j+\bar q_j&-\frac{v}{|v|}(q_j-\bar q_j)\\\frac{v}{|v|}(q_j-\bar q_j)&q_j+\bar q_j\end{matrix}\right).
\end{align*}
Hence $\gamma_j=\dfrac{1}{2}\left(\begin{matrix}q_j+\bar q_j\\\frac{v}{|v|}(q_j-\bar q_j)\end{matrix}\right)$ and $f_0^{(j)}(0)=\alpha_j=\dfrac{1}{2}(q_j+\bar q_j)$ so that $f_0$ has Taylor expansion 
\begin{equation}
f_0(z)=\dfrac{1}{2}\sum_{j=0}^\infty (q_j+\bar q_j)\dfrac{z^j}{j!}, \qquad |z| < 1.\label{f_0}
\end{equation}
We similarly have $f_1^{(j)}(z)=(1+z)^{a-j}[\alpha'_j\cos (|v|\log (1+z))+\beta'_j\sin (|v|\log (1+z))]$. With $\gamma'_j=\left(\begin{matrix}\alpha'_j\\\beta'_j\end{matrix}\right)$ and $\gamma'_0=\left(\begin{matrix}0\\1\end{matrix}\right)$, we find that 
 \begin{equation}
 f_1(z)=-\frac{v}{2|v|}\sum_{j=0}^\infty(q_j-\bar q_j)\frac{z^j}{j!},\label{f_1}
 \end{equation}
for all $|z| < 1$. Combining (\ref{f_0}) and (\ref{f_1}) gives
\be
f(z)=f_0(z)+\frac{v}{|v|}f_1(z)=\sum_{j=0}^\infty\binom{q}{j}z^j,\label{f}
\ee
for all $|z| <1$. 

Now, we have the following estimate as $q\to \infty$ along any line connecting the origin with $\infty$:
\begin{align*}
\left| \sum_{j=0}^\infty\binom{q}{j}\right| &\leq \sum_{j=0}^\infty \left|\binom{q}{j}\right| = \sum_{j=0}^\infty\left|\frac{\Gamma (q+1)}{\Gamma (q-j+1)}\right|\,\frac{1}{j!}\\
& = \sum_{j=0}^\infty\left|q^j (1 + \cO(q^{-1}))\right|\,\frac{1}{j!} \leq c\,e^{|q|},
\end{align*}
where we have used the estimate (\ref{Gamma quotient}). Here $c$ denotes a positive constant. Hence, in \eqref{f} we can let $|z|\to 1^-$. 
\end{proof}

\begin{remark} As mentioned in the above proof,  the matrix $A$ may be diagonalized over ${\mathbb C}$. The complex eigenvalues of $A$ are $w=a+i|v|$ with associated eigenvector $\left(\begin{matrix}1\\i\end{matrix}\right)$ and $\overline{w}=a-i|v|$ with associated eigenvector $\left(\begin{matrix}1\\-i\end{matrix}\right)$. We then compute the product $\prod_{\ell =0}^{j-1} (A-\ell I)$ to be 
$$\prod_{\ell =0}^{j-1}(A-\ell I)=\left(\begin{matrix}\Re (w_j)&\Im (w_j)\\-\Im (w_j)&\Re (w_j)\end{matrix}\right)$$
where $w_j$ is the Pochhammer product associated with the complex number $w$. As a consequence we have the expansion 
\begin{equation}
(1+z)^q=\sum_{j=0}^\infty\bigg[\Re (w_j)+\frac{v}{|v|}\Im (w_j)\bigg]\frac{z^j}{j!}\label{complex exp}
\end{equation}
and comparing (\ref{complex exp}) with (\ref{qbin})  we have
\begin{equation}
q_j=\Re (w_j)+\frac{v}{|v|}\Im (w_j)\label{poch eq}
\end{equation}
where $q=a+v\in{\mathbb H}_{\mathbb R}$ and $w=a+i|v|\in{\mathbb C}$. 
We note that (\ref{poch eq}) may be used to compute the quaternionic Pochhammer symbol $q_j$ using complex algebra only.
\end{remark}

\begin{corollary}
For $q\in \bH_{\mathbb R}$ with $\Sc (q) > 0$, 
\[
2^q = \sum_{j=0}^\infty\binom{q}{j}\quad\text{and}\quad 0 = \sum_{j=0}^\infty (-1)^j \binom{q}{j}.
\]
\end{corollary}
\section{Time Domain Representation of Quaternionic B-Splines}
In this section, we derive the time domain representation of a quaternionic B-spline and present an alternative way of defining quaternionic B-splines. For this purpose, given the ordinary univariate differentiation operator $\sD: C^1\to C$, we define the \emph{antidifferentiation} or \emph{integration operator of quaternionic order $q$} on the Schwartz space $\cS(\R)$ by
\be
\sD^{-q} := \frac{1}{\Gamma (q)}\,\int_0^\infty t^q e^{- t \sD}\,\frac{dt}{t} =  \frac{1}{\Gamma (q)}\,\int_0^\infty t^q T_t\,\frac{dt}{t},
\ee
where  $T_t$ the translation operator $T_t f := f(\mydot - t)$.
\begin{theorem}\label{thm1}
The quaternionic B-spline $B_q$ $(q\in{\mathbb H}_{\mathbb R})$ defined in \eqref{qB} with $\Sc (q) > 1$ has the time domain representation
\begin{equation}
B_q (t) =  \frac{1}{\Gamma (q)}\,\sum_{k=0}^\infty (-1)^k \binom{q}{k} (t - k)_+^{q-1}, \quad t\in \R.\label{time domain}
\end{equation}
This equality holds in the sense of distributions and in $L^2 (\R)$. Moreover, the quaternionic B-spline satisfies the distributional differential equation
\be\label{distqB}
\sD^q B_q = \sum_{k=0}^\infty \binom{q}{k} (-1)^k\,\delta_k,
\ee
where $\delta_k$ denotes the Dirac delta distribution supported on $k\in \Z^+_0$.
\end{theorem}
\begin{proof}
Let a distribution $B_q$ be defined by \eqref{distqB}. Then 
\begin{align*}
B_q (t) &= \sum_{k=0}^\infty \binom{q}{k} (-1)^k\,\sD^{-q} \delta_k (t)\\
&= \frac{1}{\Gamma (q)}\,\sum_{k=0}^\infty \binom{q}{k} (-1)^k\,\int_0^\infty x^q \delta_k (t-x)\,\frac{dx}{x}\\
& = \frac{1}{\Gamma (q)}\,\sum_{k=0}^\infty \binom{q}{k} (-1)^k\, (t - k)_+^{q-1}, \quad t\in \R.
\end{align*}
We then find
\begin{align*}
\wh{B_q} (\xi) &= \frac{1}{\Gamma (q)}\,\sum_{k=0}^\infty \binom{q}{k}(-1)^k\,\int_\R e^{- i  t \xi}\,(t - k)_+^{q-1}\,dt\\
&= \frac{1}{\Gamma (q)}\,\sum_{k=0}^\infty \binom{q}{k} (-1)^k\,e^{- i k \xi}\int_0^\infty e^{- i  y \xi}\,y^{q-1}\,dy.
\end{align*}
Here, we used the Dominated Convergence Theorem to exchange sum and integral. Following the arguments presented in \cite[Section 2.3]{gelfand}, we obtain for the last integral in the above expression
\[
\int_0^\infty e^{- i  y \xi}\,y^{q-1}\,dy = \lim_{\tau\to 0+} \int_0^\infty e^{- i  y \xi}\,e^{-\tau y}\,y^{q-1}\,dy = \frac{\Gamma(q)}{(i\xi)^q},\quad\xi\in \R.
\]
Whence,
\[
\wh{B_q} (\xi) = \sum_{k=0}^\infty \binom{q}{k} (-1)^k\,\frac{e^{-i k \xi}}{(i\xi)^q} =  \left(\frac{1-e^{-i \xi}}{i \xi}\right)^q,
\]
by the quaternionic binomial theorem \ref{thm0}.
\end{proof}

\noindent
Figures \ref{fig1}--\ref{fig4} below show some examples of quaternionic B-splines.

\begin{figure}[h!]
\begin{center}
\includegraphics[width=5cm, height= 4cm]{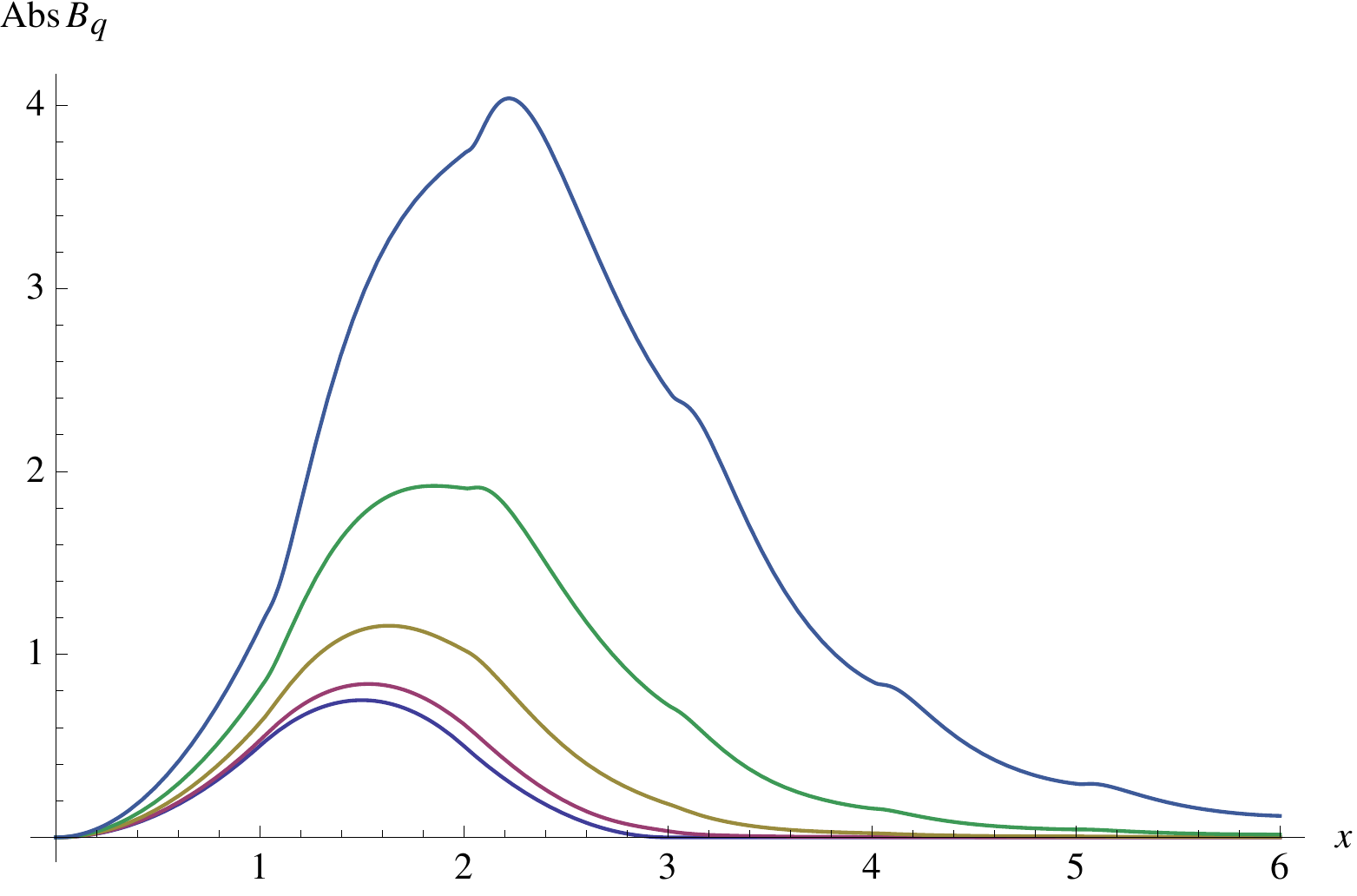}\hspace*{2cm}\includegraphics[width=5cm, height= 4cm]{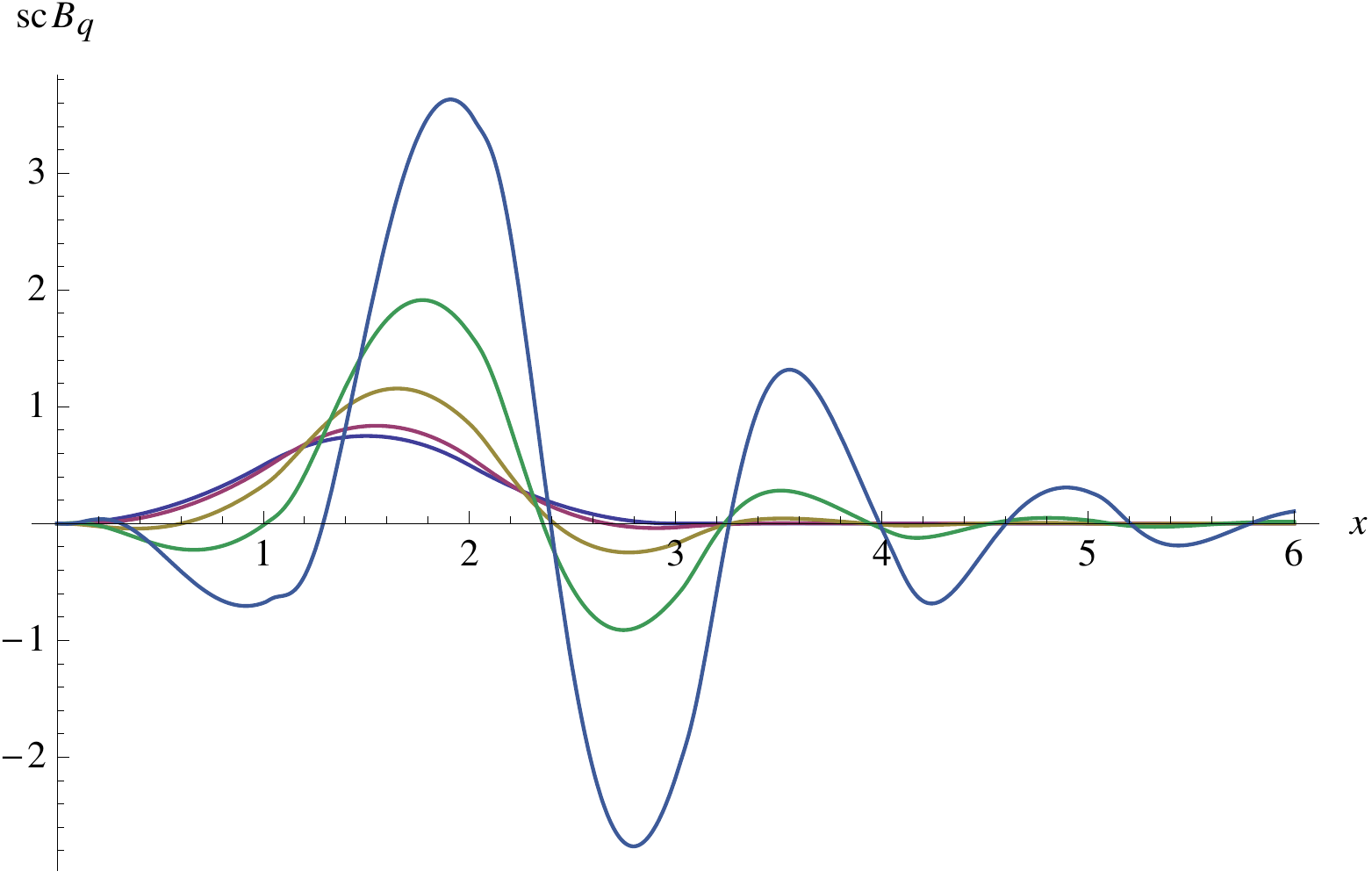}
\caption{The modulus and the scalar part of the quaternionic B-splines $B_q$ with $q = 3 + \frac{m}{5} e_1 - \frac{3m}{10} e_2 + \frac{2m}{5} e_3$, $m = 0, 1, 2, 3, 4$. The amplitudes increase with increasing $m$.}\label{fig1}
\end{center}
\end{figure}

\begin{figure}[h!]
\begin{center}
\includegraphics[width=5cm, height= 4cm]{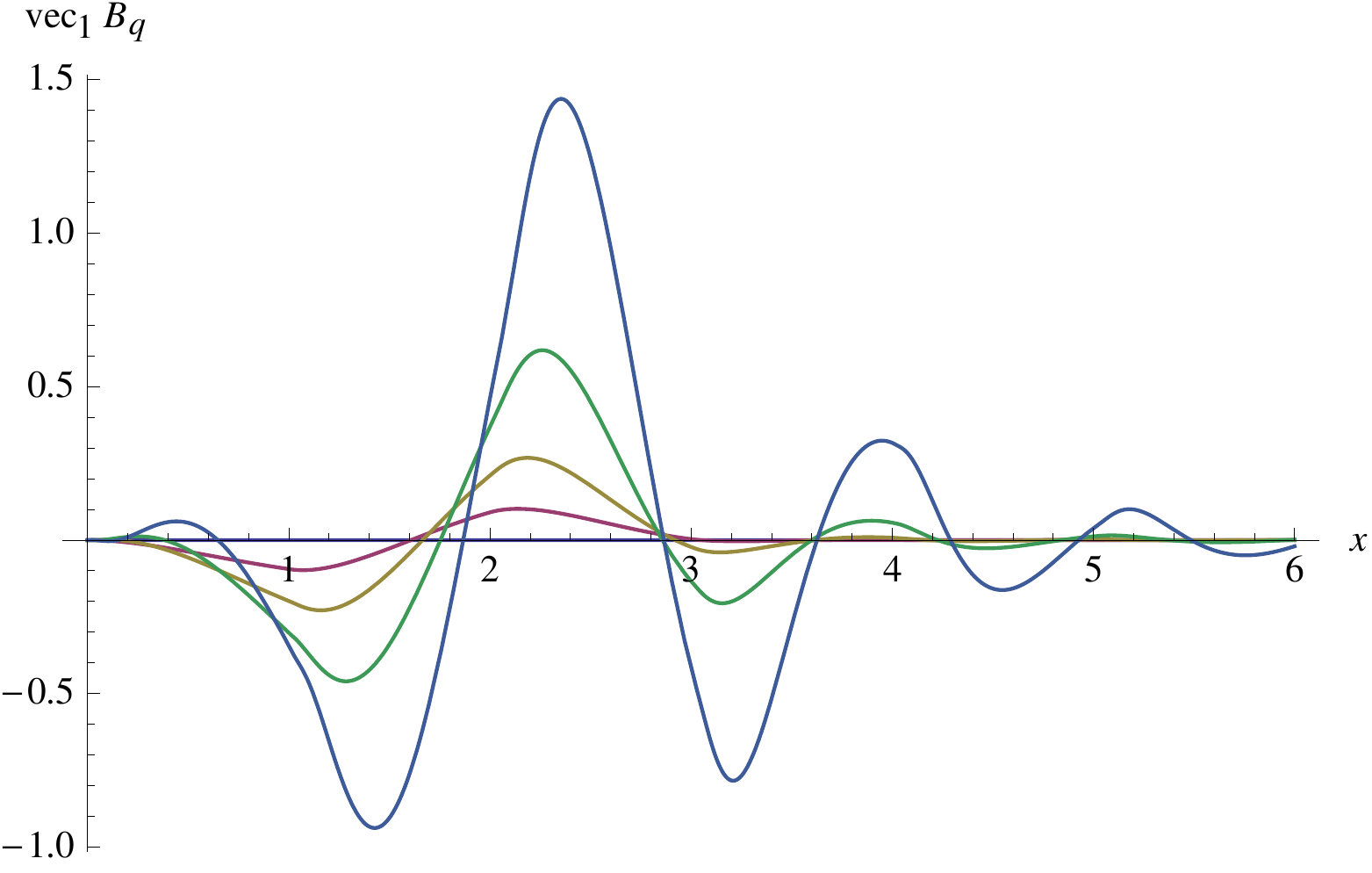}\hspace*{1cm}\includegraphics[width=5cm, height= 4cm]{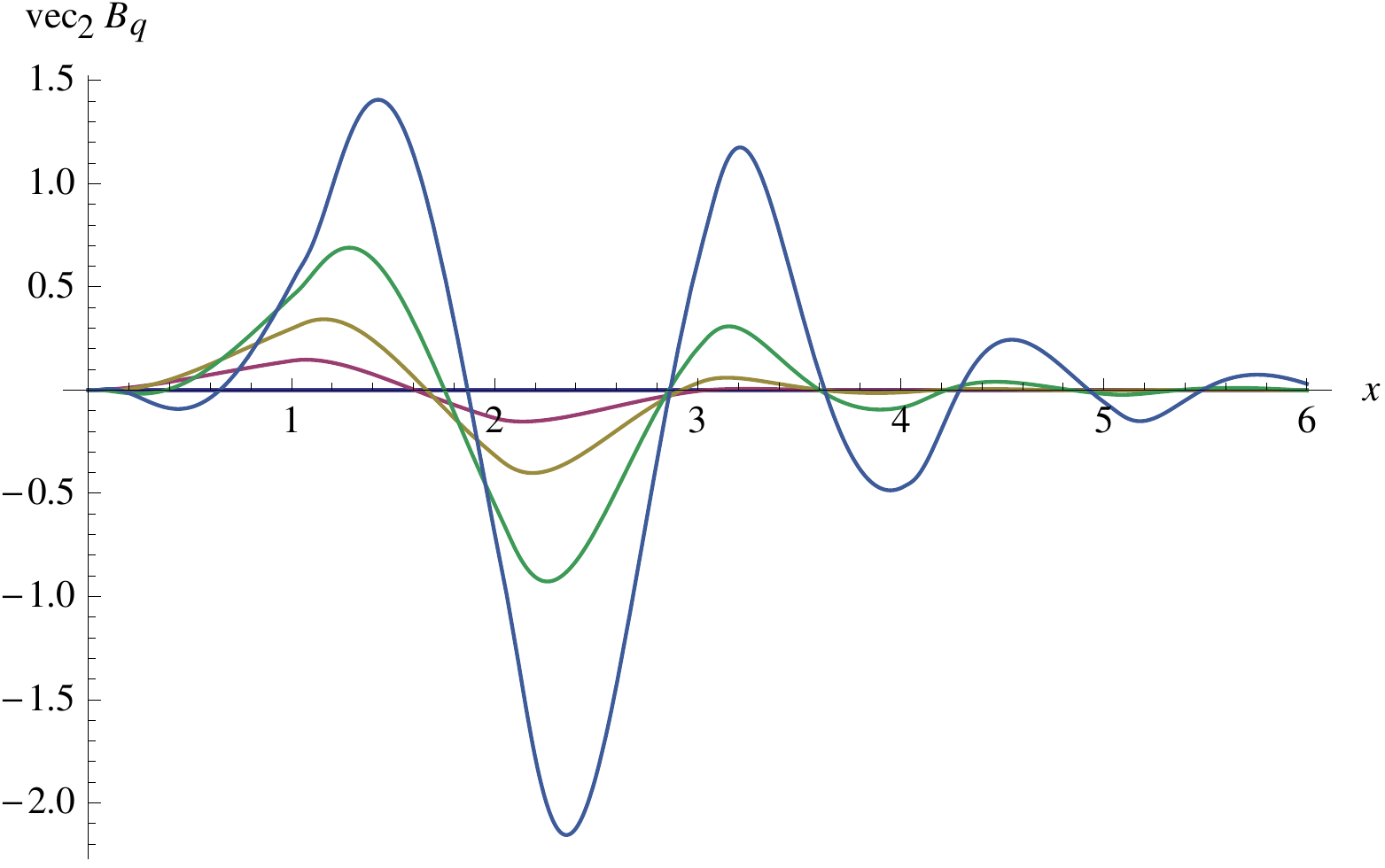}\\\includegraphics[width=5cm, height= 4cm]{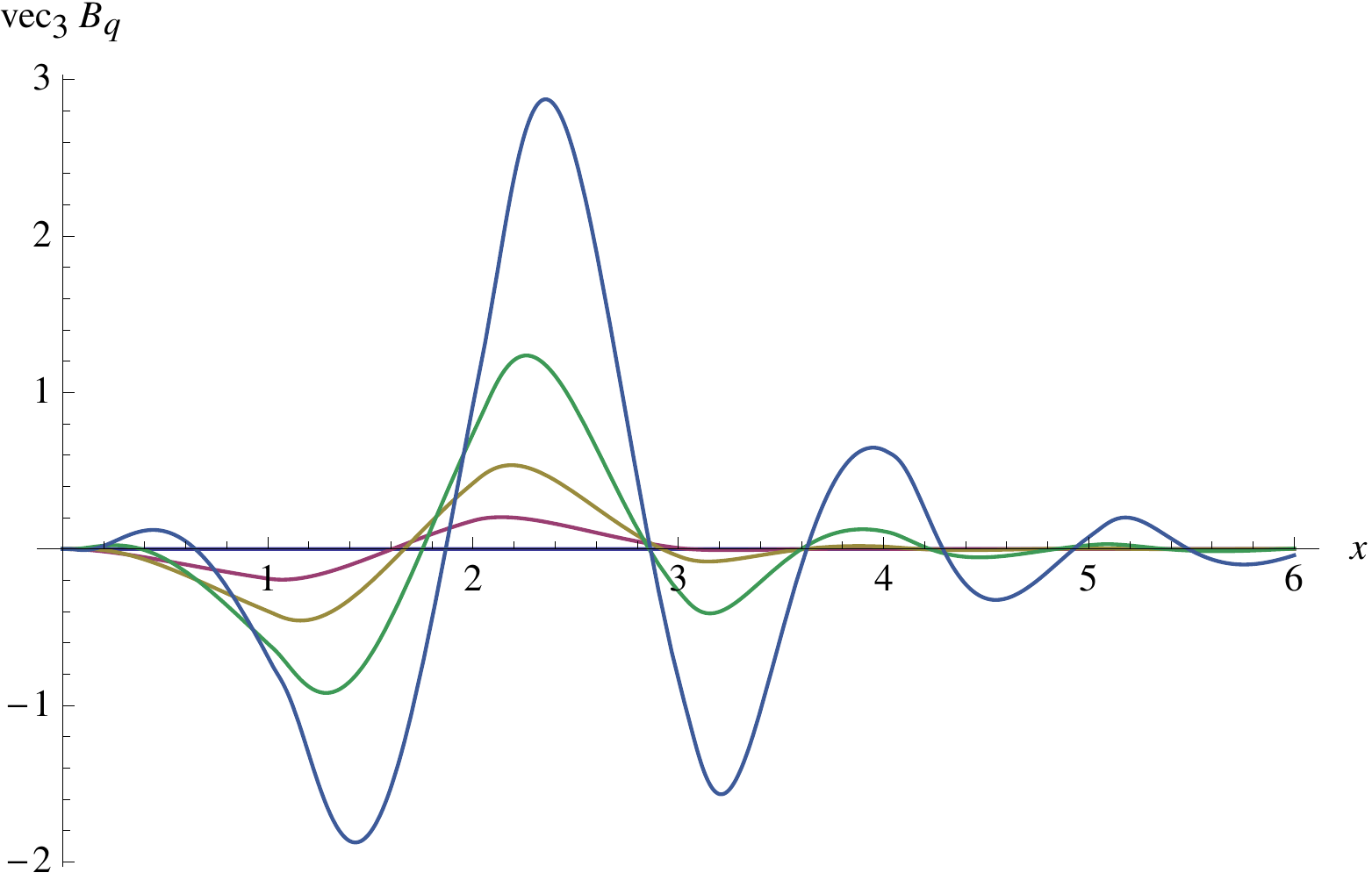}
\caption{The vector parts of the quaternionic B-splines $B_q$ with $q = 3 + \frac{m}{5} e_1 - \frac{3m}{10} e_2 + \frac{2m}{5} e_3$, $m = 0, 1, 2, 3, 4$. The amplitudes increase with increasing $m$.}\label{fig2}
\end{center}
\end{figure}

\begin{figure}[h!]
\begin{center}
\includegraphics[width=6cm, height= 5cm]{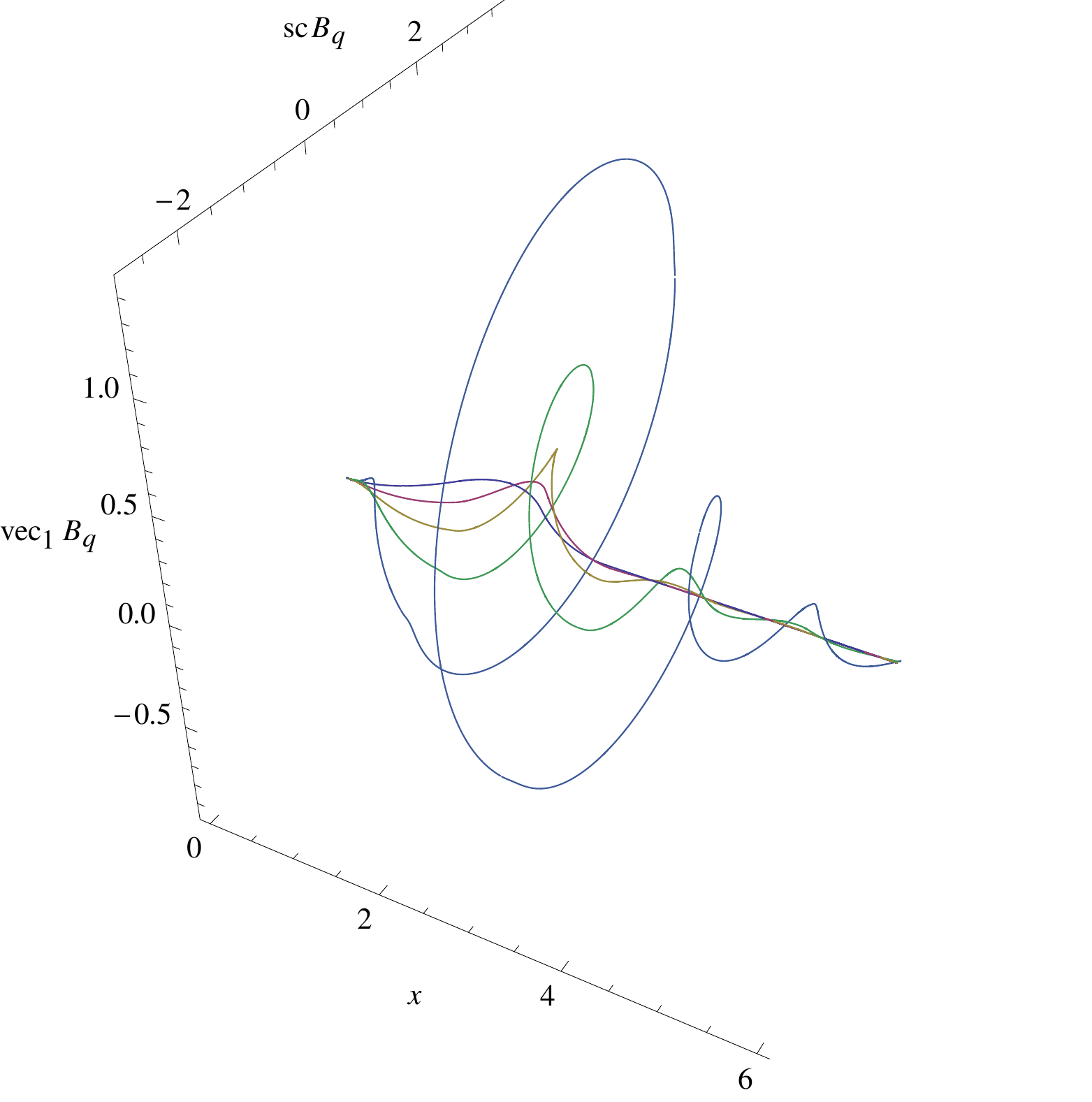}\includegraphics[width=6cm, height= 5cm]{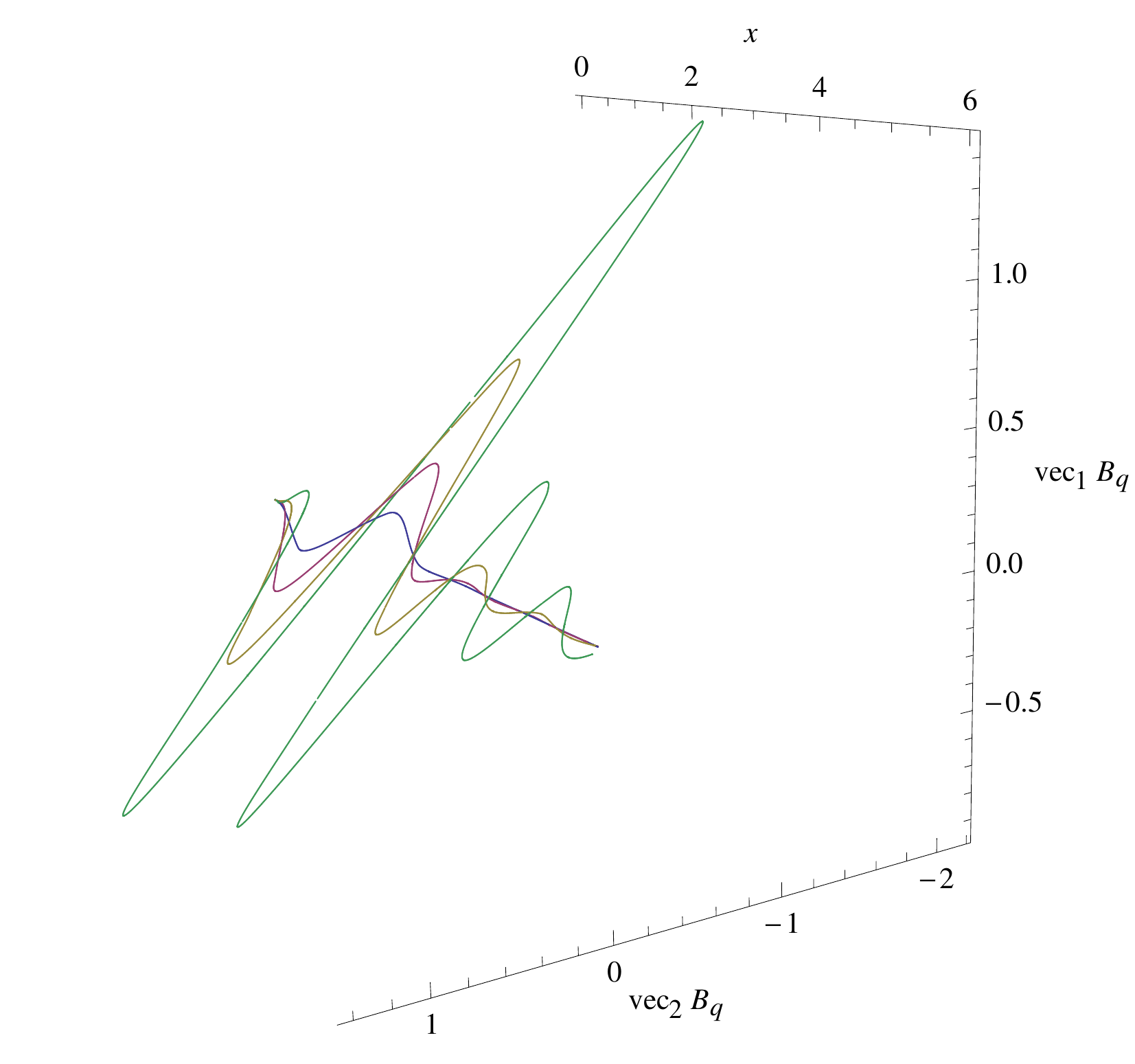}\caption{The the scalar part versus the vector part $v_1$ (left) and the vector part $v_1$ versus $v_2$ (right) of the quaternionic B-splines $B_q$ with $q = 3 + \frac{m}{5} e_1 - \frac{3m}{10} e_2 + \frac{2m}{5} e_3$, $m = 0, 1, 2, 3, 4$. The amplitudes increase with increasing $m$. Notice that in the right plot, the graphs lie in the same plane since the ratios $(\frac{m}{5})/(-\frac{3m}{5})$ are independent of $m$.}\label{fig3}
\end{center}
\end{figure}

\begin{figure}[h!]
\begin{center}
\includegraphics[width=7cm, height= 6cm]{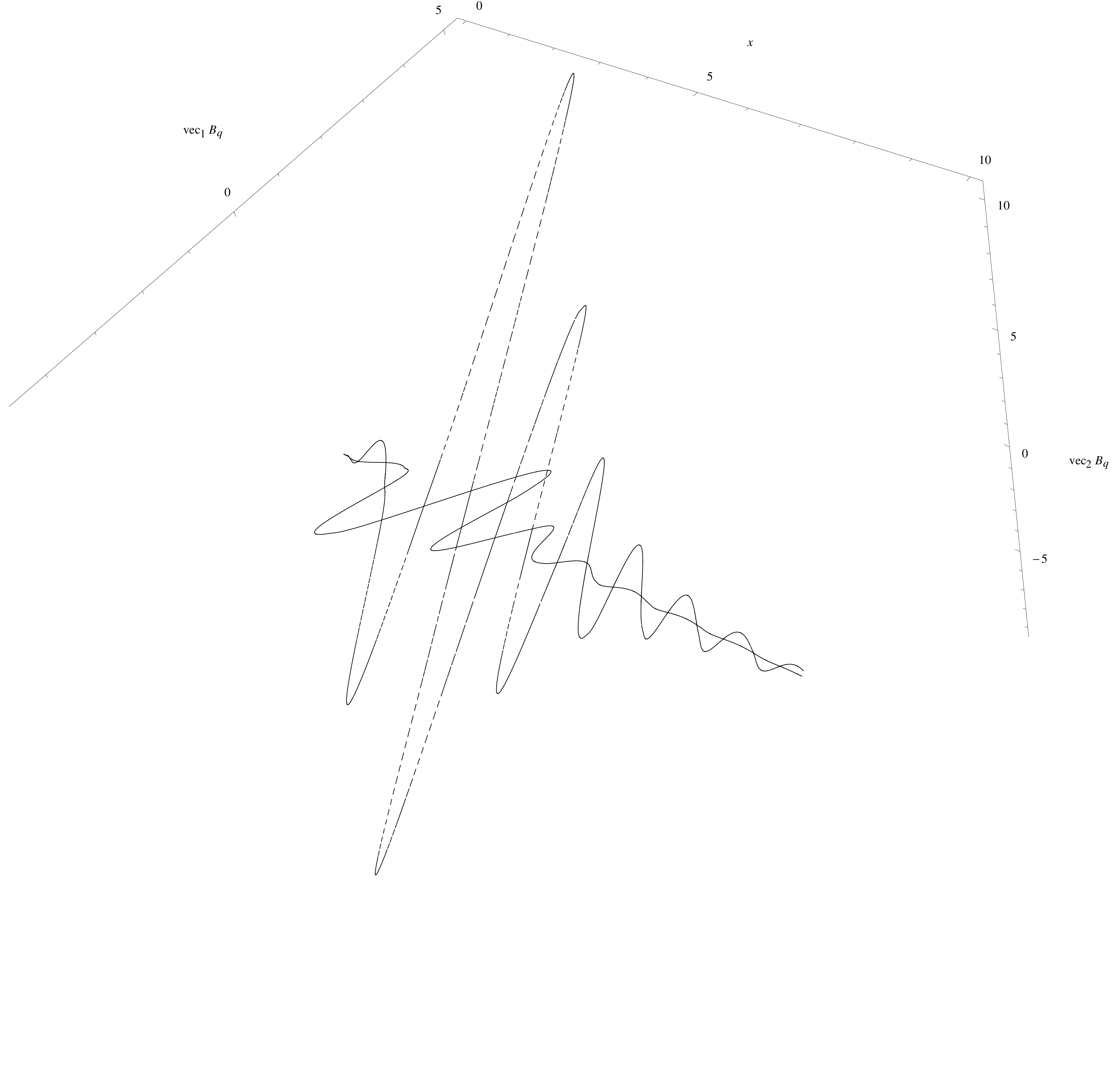}
\caption{The vector parts $v_1$ versus $v_2$ of the quaternionic B-splines $B_{q_1}$ with $q_1 = 3 - e_1 + e_2 + 2 e_3$ (solid) and  $B_{q_2}$ with $q_2 = 3 + e_1 + 2 e_2 + 2 e_3$ (dashed).}\label{fig4}
\end{center}
\end{figure}

The above proof suggests the definition of a backwards difference operator of quaternionic order $q$. Indeed, let $f\in C_c (\R)$ be a compact supported continuous function. Then we call for a fixed $q\in \bH_{\mathbb R}$,
\begin{gather*}
\nabla^q: C_c(\R)\to C_c (\R)\\
f\mapsto \nabla^q f := \sum_{k=0}^\infty \binom{q}{k} (-1)^k\,f(\mydot - k),
\end{gather*}
a \emph{backwards difference operator of quaternionic order $q$} or, for short, a \emph{quaternionic backwards difference operator}. The proof of Theorem \ref{thm1} shows then that
\[
(\nabla^q t_+^{q-1})^{\wedge} = \Gamma (q) \wh{B_q},
\]
or, equivalently,
\[
B_q = \frac{1}{\Gamma (q)}\,\nabla^q t_+^{q-1},
\]
in analogy with the corresponding property for the complex and classical polynomial B-splines.

For $j = 1, 2$, let $w_j\in {\mathbb C}$ with $\Re (w_j) > 1$.  In \cite{FBU} it is shown that 
\[
{B_{w_1}}*{B_{w_2}} ={B_{w_1+w_2}}
\]
where $B_{w_j}$ $(j=1,2)$ are complex B-splines. This is a consequence of the fact that $z^{w_1}z^{w_2}=z^{w_1+w_2}$ for complex $z$ in a neighborhood of the identity. Proposition \ref{prop1} shows that the corresponding result fails when $w_1,w_2$ are replaced by quaternions. In fact, if $q_j=a_j+v_j\in{\mathbb H}_{\mathbb R}$ with $a_j>0$ then
$$z^{q_1}z^{q_2}=z^{q_3}$$
for some $q_3\in{\mathbb H}_{\mathbb R}$ and all $z$ in a neighborhood of the identity in the complex plane if and only if $q_3=q_1+q_2$ and the vectors $v_1,v_2$ are linearly dependent in ${\mathbb R}^3$. We therefore have the following  result for quaternionic B-splines:


\begin{proposition} Let $q_i=a_i+v_i$ $(i=1,2)$ be real quaternions. Then there is a real quaternion $q_3=a_3+v_3$ for which 
\begin{equation}
B_{q_1}*B_{q_2}=B_{q_3}\label{conv formula}
\end{equation}
if and only if $q_3=q_1+q_2$ and the vectors $v_1,v_2$ are linearly dependent in ${\mathbb R}^3$.
\end{proposition}

We remark that  since ${\mathbb C}$ may be realized as a commutative subalgebra of ${\mathbb H}_{\mathbb R}$ as
$${\mathbb  C}\simeq V_k=\text{sp}_{\mathbb R}\{1,e_k\}$$
for each $k\in\{1,2,3\}$, in the case where $q_1=a_1+\lambda_1e_k$, $q_2=a_2+\lambda_2 e_k$, and $a_1,a_2> 0$, $\lambda_1,\lambda_2\in{\mathbb R}$ and $k\in\{1,2,3\}$ fixed, we recover the convolution property (\ref{conv formula}).

Given $z,w\in{\mathbb C}$ and $q_1=a_1+v_1$, $q_2=a_2+v_2\in{\mathbb H}_{\mathbb R}$, note that if $\{v_1,v_2\}$ is linearly dependent then $z^{q_1}w^{q_2}=w^{q_2}z^{q_1}$. The operators $\sD^{q_1}$ and $\nabla^{q_2}$ are multiplier operators in the sense that 
$$\widehat{\sD^{q_1}f}(\xi )=(-i\xi )^{q_1}\widehat f(\xi );\qquad \widehat{\nabla^{q_2}f}(\xi )=(1-e^{-i\xi})^{q_2}\widehat f(\xi )$$
for sufficiently smooth functions $f$. Consequently, if $\{v_1,v_2\}$ is linearly dependent, we have
\begin{equation}
\sD^{q_1}\nabla^{q_2}=\nabla^{q_2}\sD^{q_1}.\label{comm}
\end{equation}
The time domain representation (\ref{time domain}) of $B_{q_2}$ may be written as
\[
B_{q_2} = \nabla^{q_2} \sD^{-q_2} \delta.
\]
Then, provided that $a_2 - a_1 > 1$ and $\{v_1,v_2\}$ is linearly dependent, an application on (\ref{comm}) gives
\begin{align}
\sD^{q_1} B_{q_2}  &= \sD^{q_1} \nabla^{q_2} D^{-q_2} \delta =  \nabla^{q_2} \sD^{-(q_2-q_1)} \delta = \nabla^{q_1} \nabla^{q_2 - q_1} \sD^{-(q_2-q_1)}\delta\nonumber \\
& = \nabla^{q_1} B_{q_2 - q_1}.
\end{align}
For complex $q_j$, this identity reduces to the one known from the theory of complex B-splines.

\begin{proposition}
The quaternionic B-spline $B_q$, with $q\in{\mathbb H}_{\mathbb R}$ and $\Sc (q) > 1$, satisfies the recursion relation
\[
(q-1) B_q (t) = t B_{q-1} + (q-t) B_{q-1} (t - 1).
\]
\end{proposition}

\begin{proof}
Let $f$ be a function defined on $[0, \infty)$. Consider 
\begin{align*}
\nabla^q (t f(t)) &= \sum_{k=0}^\infty (-1)^k \binom{q}{k} (t-k) f(t-k)\\
& = t \nabla^q f(t) - \sum_{k=1}^\infty (-1)^k \binom{q}{k}\, k\, f(t-k)\\
& =  t \nabla^q f(t) - \sum_{k=1}^\infty (-1)^k \binom{q-1}{k-1}\, q\, f(t-k)\\
& = t \nabla^q f(t) - q\left[\sum_{k=1}^\infty (-1)^k \binom{q}{k} f(t-k) - \sum_{k=1}^\infty (-1)^k \binom{q-1}{k} f(t-k)\right]\\
& = t \nabla^q f(t) - q \nabla^q f(t) + q \nabla^{q-1} f(t)\\
& = (t-q) \nabla^q f(t)  + q \nabla^{q-1} f(t).
\end{align*}

Now let $f(t) := \dfrac{t_+^{q-2}}{\Gamma (q)}$. Then
\begin{align*}
\nabla^q \frac{t_+^{q-1}}{\Gamma (q)} &= (t-q) \nabla^q \frac{t_+^{q-2}}{\Gamma (q)}  + q \nabla^{q-1}\frac{t_+^{q-2}}{\Gamma (q)}\\
& = \frac{t-q}{q-1}\nabla \nabla^{q-1} \frac{t_+^{q-2}}{\Gamma (q-1)} + \frac{1}{q-1}\nabla^{q-1} \frac{t_+^{q-2}}{\Gamma (q-1)}\\
& = \frac{t-q}{q-1}\nabla B_{q-1} + \frac{1}{q-1}\nabla^{q-1} B_{q-1}.
\end{align*}
As $\nabla B_{q-1} (t) = B_{q-1} (t) - B_{q-1}(t-1)$, substitution into the above equation yields the result.
\end{proof}
\section{Miscellaneous properties of Quaternionic Gamma functions and B-splines}
In this section, we show that the quaternionic Gamma function is rotationally covariant and satisfies a certain homogeneity property. Both properties are inherited by the quaternionic B-splines.

To this end, let $\sigma \in\text{SO}(3)$ be a rotation. Then we define $1\otimes\sigma\in\text{SO}(4)$ as the rotation of ${\mathbb R}^4$ which fixes the $x_1$ axis  and rotates the three-dimensional orthogonal subspace via $\sigma$, i.e., in quaternionic notation, if $q=a+v\in{\mathbb H}_{\mathbb R}$ then  $(1\otimes\sigma)(a+v)=a+\sigma v$.

\begin{lemma}
\label{rot cov} Let $\sigma\in\text{SO}(3)$ and $q\in{\mathbb H}_{\mathbb R}$. Then 
$$\Gamma ((1\otimes\sigma )(q))=(1\otimes\sigma)(\Gamma (q)).$$
\begin{proof} Since $\Gamma (q)$ has the integral representation (\ref{gamma})
we obtain
\begin{align*}
\Gamma ((1\otimes\sigma)(q))&=\Gamma (a+\sigma (v))\\
&=\int_0^\infty t^{a-1}\cos (|\sigma v|\log t)\, dt+\frac{\sigma v}{|\sigma v|}\int_0^\infty t^{a-1}\sin (|\sigma v|\log t)\, dt\\
&=\int_0^\infty t^{a-1}\cos (|v|\log t)\, dt+\frac{\sigma v}{|v|}\int_0^\infty t^{a-1}\sin (|v|\log t)\, dt\\
&=(1\otimes\sigma )(\Gamma (q)).\qedhere
\end{align*}
\end{proof}
\end{lemma}
As $\Gamma (q)\in{\mathbb H}_{\mathbb R}$, we write 
\[
\Gamma (q)=(\Gamma (q))_0+\sum_{i=0}^3e_i(\Gamma (q))_i. 
\]
However, since $\Gamma (q)_i=\dfrac{v_i}{|v|}\int\limits_0^\infty t^{a-1}\sin (|v|\log t)\, dt$ for $i\in\{1,2,3\}$, we have the following homogeneity relation.
\begin{lemma} 
For all $q=a+v\in{\mathbb H}$, $i\in\{1,2,3\}$,
$$v_j(\Gamma (q))_i=v_i(\Gamma (q))_j.$$
\end{lemma}
The rotation-covariance of the Gamma function is inherited by the quaternionic B-splines. In fact, we have
\begin{lemma}
Let $\sigma\in\text{SO}(3)$. Then for all $q\in{\mathbb H}_{\mathbb R}$,
$$B_{(1\otimes\sigma)(q)}(t)=(1\otimes\sigma )(B_q(t)).$$
\end{lemma}
Furthermore, the homogeneity relation for the quaternionic Gamma function also transfers to quaternionic B-splines.
\begin{lemma}
For all $q=a+v\in{\mathbb H}_{\mathbb R}$, $i\in\{1,2,3\}$,
$$v_j(B_q(t))_i=v_i(B_q(t))_j.$$
\end{lemma}
\begin{proof}
Note that the Fourier transform of $B_q$ $(q=a+v)$ has the representation (\ref{Bhat}). 
Consequently, for each $i\in\{1,2,3\}$,
$$(B_q(t))_i=\frac{v_i}{|v|}\int_{-\infty}^\infty\Xi (\xi )^a\sin (|v|\log\Xi (\xi ))e^{i\xi t}\, d\xi$$
from which the result follows.
\end{proof}
\section{Refinability and Multiscale Structure}
In this section we show refinability and multiscale structure of quaternionic B-splines. To this end, define for a fixed $q\in \bH_{\mathbb R}$ with $\Sc q > 1$,
\begin{align*}
H_0 (\xi) := \frac{\wh{B_q}(2\xi)}{\wh{B_q}(\xi)}, \quad \xi\in \R.
\end{align*}
Using Definition \eqref{eq2.1} and the quaternionic binomial theorem \ref{thm0}, one shows that
\begin{align*}
H_0 (\xi) = \frac{(1-e^{-2 i \xi})^q}{(1-e^{-i \xi})^q}\,\frac{(i \xi)^q}{(2 i \xi)^q} = \frac{1}{2^q}\,(1+e^{-i \xi})^q = \frac{1}{2^q}\,\sum_{k=0}^\infty \binom{q}{k}\,e^{-i k \xi}.
\end{align*}
Hence, $H_0$ is a quaternion-valued $2\pi$-periodic bounded function. In other words, the quaternionic B-spline $B_q$ satisfies the refinement equation
\be\label{eq5.1}
B_q (t) = \sum_{k=0}^\infty h(k) {B_q}(2t-k), \quad \text{ a.e. } t\in \R,
\ee
where the $\{h(k)\}_{k\in \Z}$ are the Fourier coefficients of the function $H_0$ and the convergence is in the sense of $L^2$. 

Recalling the integrand in \eqref{eq2.3}, we have the following inequalities for $|\wh{B_q}|^2$:
\be\label{eq5.2}
|\wh{B_a}(\xi)|^2 \leq |\wh{B_q}(\xi)|^2 \leq |\wh{B_a}(\xi)|^2\,\cosh(\pi |v|), \quad \xi\in \R.
\ee

\begin{theorem}
Suppose that $q\in \bH_{\mathbb R}$ with $\Sc q > 1$  fixed. Denote by $D$ the unitary dilation operator $(D f)(x) := \sqrt{2} f(2x)$.
Define the shift-invariant spaces
\be
V_n^q := \clos_{L^2} \Span \left\{D^nT_k B_q : k\in \Z\right\}, \quad n\in \Z.
\ee
Then $\{V_n^q\}_{n\in \Z}$ generates a dyadic multiresolution analysis of $L^2(\R,\bH_{\mathbb C})$.
\end{theorem}
\begin{proof}
We employ Theorem 2.13 from \cite{W}. Observe that 
\begin{enumerate}
\item $\wh{B_q}$ is continuous at the origin and $\wh{B_q}(0) = 1$. 
\item $B_q$ satisfies a refinement equation, namely \eqref{eq5.1}. 
\item $\{T_k B_q\}_{k\in \Z}$ is a Riesz sequence in $L^2({\mathbb R},{\mathbb H}_{\mathbb C})$. Indeed, by \eqref{eq5.2} and the fact that the fractional B-splines form a Riesz sequence for $L^2(\R,\C)$ \cite{UB}, we obtain
\begin{align*}
0 < A &\leq \sum_{k\in \Z} |\wh{B_a}(\xi+2\pi k)|^2 \leq \sum_{k\in \Z}|\wh{B_q}(\xi+2\pi k)|^2\\
& \leq \cosh(\pi |v|)\,\sum_{k\in \Z} |\wh{B_a}(\xi+2\pi k)|^2 \leq B \,\cosh(\pi |v|) < \infty,
\end{align*}
for some positive constants $A\leq B$.
\end{enumerate}
The statement now follows.
\end{proof}

Next, we consider the approximation order of the shift-invariant spaces $V_n^q$. To this end, denote by $\sP_n: L^2(\R,\bH_{\mathbb C})\to V_n^q$ the linear operator
\be\label{eq16}
\sP_n (f) := \sum_{k\in \Z} \inn{f}{D^n T_k B_q}\,D^n T_k B_q,
\ee
where $\inn{\mydot}{\mydot}$ denotes the $L^2({\mathbb R},{\mathbb H}_{\mathbb C})$ inner product.

The operator $\sP_n$ is said to provide \emph{approximation order $\alpha$} if 
\[
\Vert f - \sP_n f \Vert_{L^2} \in \mathcal{O}(2^{-n\alpha}), 
\]
for all $f\in H^\alpha (\R,\bH_{\mathbb C})$.

We require the following known result adapted to our setting. For a proof and details, see for instance \cite{DRoSh3}.
\begin{lemma}\label{lem2}
The approximation order of the operator $\sP_n$ is given by $\min\{\lceil\Sc q\rceil, m\}$, where $m$ is the order of the zero of $1 - |H_0(\xi)|^2$ at the origin.
\end{lemma}

We then arrive at the next result which provides the approximation order of the shift-invariant spaces associated with $B_q$.
\begin{theorem}
The operator $\sP_n$ defined in \eqref{eq16} provides approximation order $2$.
\end{theorem}
\begin{proof}
Note that 
\[
1 - |H_0(\xi)|^2 = 1 - (\cos\tfrac{\xi}{2})^{2 q} = \xi^2\,\cK(\xi),
\]
where $\cK$ is a function not vanishing at the origin. The result now follows from Lemma \ref{lem2}.
\end{proof}
We note here the rotation-covariance property of the inner products $\langle B_q,f\rangle$.  For $f\in L^2({\mathbb R}, {\mathbb H}_{\mathbb C})$ and $\sigma\in\text{SO}(3)$, by Lemma \ref{rot cov} we have
$$\langle B_q,(1\otimes\sigma )f\rangle =\langle (1\otimes\sigma^T)B_q,f\rangle =\langle B_{(1\otimes\sigma^T )(q)},f\rangle .$$

\section{Convergence to Gaussian Functions}
Next, we investigate the convergence of quaternionic B-splines to modulated and shifted Gaussians when $\Sc (q) \to \infty$. First, we consider the case of pointwise convergence in the Fourier domain.

\begin{theorem}\label{pointwise cvgence}
Fix a  vector $v\in \bH_{\mathbb R}$. The quaternionic B-splines $B_q$ $(q=a+v,\ a>0)$ converge pointwise in the Fourier domain to a shifted and modulated Gaussian as $a\to\infty$ in the sense that 
\[
\lim_{a\to\infty} \frac{\wh{B_q}(\xi /\sqrt a)}{e^{-i\sqrt a\xi/2}\, e^{-|v|/a}\, e^{-(\xi/\sqrt{24}+i v/\sqrt{a})^2}\,e^{-\xi^2 v/24a}} = 1.
\]
\end{theorem}
\begin{proof}
Note that by Taylor's remainder theorem, there is a constant $C$ and a function $R(\xi )$ with $|R(\xi )|\leq C|\xi |^4$ such that
$$\log\bigg(\frac{1-e^{-i\xi}}{i\xi}\bigg)=\bigg(-\frac{i\xi}{2}-\frac{\xi^2}{24}\bigg)+R(\xi ).$$
Multiplying both sides by $q=a+v$, exponentiating both sides, and replacing $\xi $ by $\xi /\sqrt a$ gives
$$\wh B_q(\xi /\sqrt a )=e^{-q(i\xi /2\sqrt a+\xi ^2/24a)}e^{R(\xi /\sqrt a)}$$
so that 
\begin{equation}
\frac{\wh{B_q}(\xi /\sqrt a)}{e^{-i\sqrt a\xi/2}e^{-\xi^2/24}e^{-i\xi v/\sqrt a}e^{-\xi^2 v/24a}}=e^{g(\xi /\sqrt a)}=1+\mathcal{O}(|\xi |^4|q|/a^2)\label{gaussian cvgence}
\end{equation}
and the result follows by allowing $a\to\infty$.
\end{proof}
We note that in the denominator on the left hand side of (\ref{gaussian cvgence}), we have
$|e^{-\xi^2v/24a}|=1$ and
\[
|e^{-i\xi v/\sqrt a}|= \bigg|\cosh \bigg(\frac{|v|\xi}{\sqrt a}\bigg)+\frac{iv}{|v|}\sinh\bigg(\frac{|v|\xi}{\sqrt a}\bigg)\bigg|=\sqrt{\cosh\bigg(\frac{2|v|\xi}{\sqrt a}\bigg)}\leq e^{|v||\xi|/\sqrt a}.\qedhere
\]
%
In order to establish convergence to a modulated and shifted Gaussian in $L^p$-norm, we require the following lemma and its corollary.
\begin{lemma}\label{lem7} 
Let $q=a+v\in {\mathbb H}_{\mathbb R}$ with $a>0$. Then we have the Fourier transform relation 
\begin{equation}
\int_{-\infty}^\infty e^{-q\xi^2}e^{it\xi}\, d\xi =\sqrt{\dfrac{2\pi}{q}}e^{-t^2/(4q)}\label{FT relation}
\end{equation}
where $\sqrt{q}=\dfrac{q+|q|}{\sqrt 2\sqrt{a+|q|}}$.
\end{lemma}

\begin{proof} Let $I(t)=\int_{-\infty}^\infty e^{-q\xi^2}e^{it\xi}\, d\xi$. Note that 
\begin{align*}
e^{-q\xi^2}=e^{-a\xi^2}e^{-v\xi^2}&=e^{-a\xi^2}[\cos (|v|\xi^2)-\frac{v}{|v|}\sin (|v|\xi^2)]\\
&=\frac{e^{-a\xi^2}}{2}[e^{i|v|\xi^2}+e^{-i|v|\xi^2}-\frac{v}{i|v|}(e^{i|v|\xi^2}-e^{-i|v|\xi^2})],
\end{align*} 
from which we see that 
\begin{equation}
I(t)=A_1+A_2-\frac{v}{i|v|}(A_1-A_2)\label{A_1+A_2}
\end{equation}
where
$A_1=\frac{1}{2}\int\limits_{-\infty}^\infty e^{-a\xi^2}e^{i|v|\xi}e^{i\xi t}\, d\xi$ and $A_2=\frac{1}{2}\int\limits_{-\infty}^\infty e^{-a\xi^2}e^{-i|v|\xi}e^{i\xi t}\, d\xi$.
Let $z=a-i|v|\in{\mathbb C}$. Then 
$A_1=\frac{1}{2}\int\limits\limits_{-\infty}^\infty e^{-z\xi^2}e^{i\xi t}\, d\xi=\frac{1}{2}\sqrt{\frac{2\pi}{z}}e^{-t^2/4z}$
since $\Re (z)=a>0$. Similarly, we have $A_2=\frac{1}{2}\sqrt{\frac{2\pi}{\overline z}}e^{-t^2/4\overline z}$. Since $\sqrt{z}=\frac{z+|z|}{\sqrt 2\sqrt{a+|z|}}$, $\sqrt{\overline z}=\frac{\overline z+|z|}{\sqrt 2\sqrt{a+|z|}}$, $\frac{1}{z+|z|}=\frac{\overline z+|z|}{2|z|(a+|z|)}$, and $\frac{1}{\overline z +|z|}=\frac{z+|z|}{2|z|(a+|z|)}$, we obtain from (\ref{A_1+A_2})
\begin{align}
I(t)&=\frac{1}{2}\bigg[\bigg(\frac{\sqrt{4\pi}\sqrt{a+|z|}}{z+|z|}e^{-\overline zt^2/4|z|^2}+\frac{\sqrt{4\pi}\sqrt{a+|z|}}{\overline z+|z|}e^{-zt^2/4|z|^2}\bigg)\notag\\
&\qquad\qquad-\frac{1}{2}\frac{v}{i|v|}\bigg(\frac{\sqrt{4\pi}\sqrt{a+|z|}}{z+|z|}e^{-\overline zt^2/4|z|^2}-\frac{\sqrt{4\pi}\sqrt{a+|z|}}{\overline z+|z|}e^{-zt^2/4|z|^2}\bigg)\bigg]\notag\\
&=\frac{\sqrt{\pi}\sqrt{a+|z|}}{2}e^{-at^2/4|z|^2}\bigg[\bigg(\frac{e^{-i|v|t^2/4|z|^2}}{z+|z|}+\frac{e^{i|v|t^2/4|z|^2}}{\overline z+|z|}\bigg)-\frac{v}{i|v|}\bigg(\frac{e^{-i|v|t^2/4|z|^2}}{z+|z|}+\frac{e^{i|v|t^2/4|z|^2}}{\overline z+|z|}\bigg)\bigg]\notag\\
&=\frac{\sqrt{\pi}}{2|z|\sqrt{a+|z|}}e^{-at^2/4|z|^2}\bigg[(\overline z +|z|)e^{-i|v|t^2/4|z|^2}+(z+|z|)e^{i|v|t^2/4|z|^2}\notag\\
&\qquad\qquad\qquad\qquad\qquad\qquad\qquad -\frac{v}{i|v|}\bigg((\overline z+|z|)e^{-i|v|t^2/4|z|^2}-(z+|z|)e^{i|v|t^2/4|z|^2}\bigg)\bigg]\notag\\
&=\frac{\sqrt{\pi}}{2|z|\sqrt{a+|z|}}e^{-at^2/4|z|^2}\bigg[(a+i|v|+|z|)\bigg(\cos\bigg(\frac{|v|t^2}{4|z|^2}\bigg)-i\sin\bigg(\frac{|v|t^2}{4|z|^2}\bigg)\bigg)\notag\\
&\qquad\qquad\qquad\qquad -(a-i|v|+|z|)\bigg(\cos\bigg(\frac{|v|t^2}{4|z|^2}\bigg)+i\sin\bigg(\frac{|v|t^2}{4|z|^2}\bigg)\bigg)\notag\\
&\qquad\qquad\qquad\qquad -\frac{v}{i|v|}\bigg[(a+i|v|+|z|)\bigg(\cos\bigg(\frac{|v|t^2}{4|z|^2}\bigg)-i\sin\bigg(\frac{|v|t^2}{4|z|^2}\bigg)\bigg)\notag\\
&\qquad\qquad\qquad\qquad\qquad -(a-i|v|+|z|)\bigg(\cos\bigg(\frac{|v|t^2}{4|z|^2}\bigg)+i\sin\bigg(\frac{|v|t^2}{4|z|^2}\bigg)\bigg)\bigg]\bigg]\notag\\
&=\frac{\sqrt\pi}{|z|\sqrt{a+|z|}}e^{-at^2/4|z|^2}\bigg[(a+|z|)\cos\bigg(\frac{|v|t^2}{4|z|^2}\bigg)+|v|\sin \bigg(\frac{|v|t^2}{4|z|^2}\bigg)\notag\\
&\qquad\qquad\qquad\qquad\qquad -\frac{v}{|v|}\bigg[|v|\cos \bigg(\frac{|v|t^2}{4|z|^2}\bigg)-(a+|z|)\sin\bigg(\frac{|v|t^2}{4|z|^2}\bigg)\bigg]\bigg]\notag\\
&=\frac{\sqrt{\pi}}{|z|\sqrt{a+|z|}}e^{-at^2/4|z|^2}(a-v+|z|)\bigg[\cos\bigg(\frac{|v|t^2}{4|z|^2}\bigg)+\frac{v}{|v|}\sin\bigg(\frac{|v|t^2}{4|z|^2}\bigg)\bigg].\label{I comp1}
\end{align}
On the other hand, since $|q|=|z|$, the right hand side of (\ref{FT relation}) is equal to
\begin{align}
\sqrt{\frac{2\pi}{q}}e^{-t^2/4q}&=\frac{\sqrt{2\pi}\sqrt{2}\sqrt{a+|q|}}{q+\overline q}e^{-\overline qt^2/4|q|^2}\notag\\
&=\frac{\sqrt{4\pi}\sqrt{a+|q|}(\overline q+|q|)}{(q+|q|)(\overline q+|q|)}e^{-at^2/4|q|^2}e^{vt^2/4|q|^2}\notag\\
&=\frac{\sqrt{\pi}}{|z|\sqrt{a+|z|}}e^{-at^2/4|z|^2}(a-v+|z|)e^{vt^2/4|z|^2}\notag\\
&=\frac{\sqrt{\pi}}{|z|\sqrt{a+|z|}}e^{-at^2/4|z|^2}(a-v+|z|)\bigg[\cos\bigg(\frac{|v|t^2}{4|z|^2}\bigg)+\frac{v}{|v|}\sin\bigg(\frac{|v|t^2}{4|z|^2}\bigg)\bigg]
\label{I comp2}
\end{align}
Comparing (\ref{I comp1}) and (\ref{I comp2}) gives the result.
\end{proof}

\begin{corollary} \label{cor2}
If $q=a+v\in{\mathbb H}$ with $a>0$ and $\alpha\in{\mathbb R}$, then 
$$\int_{-\infty}^\infty e^{-q\xi^2}e^{-i\alpha q\xi}e^{i\xi t}\, d\xi =\sqrt{\frac{2\pi}{q}}e^{-q\alpha^2/4}e^{-\alpha t/2}e^{-t^2/4q}.$$
\end{corollary}
\begin{proof}
Let $J(t)= \int\limits_{-\infty}^\infty e^{-q\xi^2}e^{-i\alpha q\xi}e^{i\xi t}\, d\xi$. Then, since
\begin{align*}
e^{-q\xi^2}&=e^{-a\xi^2}e^{-v\xi^2}=e^{-a\xi^2}(\cos (|v|\xi^2)-\frac{v}{|v|}\sin (|v|\xi^2)),\\
e^{-i\alpha q\xi}&=e^{-i\alpha a\xi}e^{-i\alpha v\xi}=e^{-i\alpha a\xi}(\cosh (\alpha |v|\xi )-\frac{iv}{|v|}\sinh (\alpha |v|\xi )),
\end{align*}
we have
\begin{align}
J(t)&=\int_{-\infty}^\infty e^{-a\xi^2}e^{-i\alpha a\xi}(\cos (|v|\xi^2)-\frac{v}{|v|}\sin (|v|\xi^2))(\cosh (\alpha |v|\xi )-\frac{iv}{|v|}\sinh (\alpha |v|\xi ))e^{i\xi t}\, d\xi\notag\\
&=\int_{-\infty}^\infty e^{-a\xi^2}e^{-i\alpha a\xi}[\cos (|v|\xi^2)\cosh(\alpha |v|\xi )-i\sin (|v|\xi ^2)\sinh (\alpha |v|\xi)\notag\\
&\qquad\qquad\qquad\qquad -\frac{v}{|v|}(\sin(|v|\xi^2)\cosh(\alpha |v|\xi)+i\cos (|v|\xi^2)\sinh (\alpha |v|\xi ))]e^{i\xi t}\, d\xi .\label{J int1}
\end{align}
Applying the relations
\begin{align*}
\cos (\alpha )\cosh (\beta)-i\sin (\alpha )\sinh (\beta )&=\cos (\alpha +i\beta ),\\
\sin (\alpha )\cosh (\beta )+i\cos (\alpha )\sinh (\beta )&=-\sin (\alpha +i\beta),
\end{align*}
to (\ref{J int1}) yields
\begin{align*}
J(t)&=\int_{-\infty}^\infty e^{-a(\xi^2+i\alpha\xi )}\cos (|v|(\xi^2+i\alpha\xi ))e^{i\xi t}\, d\xi\\
&\qquad\qquad -\frac{v}{|v|}\int_{-\infty}^\infty e^{-a(\xi ^2+i\alpha\xi )}\sin (|v|(\xi^2+i\alpha\xi ))e^{i\xi t}\, d\xi\\
&=e^{-a\alpha^2/4}\int_{-\infty}^\infty e^{-a(\xi +i\alpha /2)^2}\cos (|v|((\xi +i\alpha /2)^2+\alpha^2/4))e^{i\xi t}\, d\xi\\
&\qquad\qquad -e^{-a\alpha^2/4}\frac{v}{|v|}\int_{-\infty}^\infty e^{-a(\xi +i\alpha /2)^2}\sin (|v|((\xi +i\alpha /2)^2+\alpha^2/4))e^{i\xi t}\, d\xi\\
&=e^{-a\alpha^2/4}\int_{-\infty}^\infty e^{-a\xi^2}\cos (|v|(\xi^2+\alpha^2/4))e^{i(\xi-i\alpha /2)t}\, d\xi\\
&\qquad\qquad -\frac{v}{|v|}e^{-a\alpha^2/4}\int_{-\infty}^\infty e^{-a\xi^2}\sin (|v|(\xi^2+\alpha^2/4))e^{i(\xi-i\alpha /2)t}\, d\xi
\end{align*}
where we have applied the complex change of variables $\xi\to \xi -i\alpha /2$. This requires a change of the contour of integration. However, in this situation, standard techniques of compiex analysis may be applied to show  that the contour remains unchanged. Hence we have
\begin{align*}
J(t)&=e^{-a\alpha^2/4}e^{\alpha t/2}\int_{-\infty}^\infty e^{-a\xi^2}[\cos (|v|(\xi^2+\alpha^2/4))-\frac{v}{|v|}\sin (|v|(\xi^2+\alpha^2/4))]e^{i\xi t}\, d\xi\\
&=e^{-a\alpha^2/4}e^{\alpha t/2}\int_{-\infty}^\infty e^{-a\xi^2}e^{-v(\xi^2+\alpha^2/4)}e^{i\xi t}\, d\xi\\
&=e^{-q\alpha^2/4}e^{\alpha t/2}\int_{-\infty}^\infty e^{-q\xi^2}e^{i\xi t}\, dt=e^{-q\alpha^2/4}e^{\alpha t/2}\sqrt{\frac{2\pi}{q}}e^{-t^2/4q},
\end{align*}
where in the last line we have applied Lemma \ref{lem7}.
\end{proof} 
For later purposes, we require a special form for the integrand in Corollary \ref{cor2}. To this end, let $q=a+v$, $q'=\dfrac{q}{24a}$ (which has real part $\dfrac{1}{24}>0$) and $\alpha =12\sqrt a$.
Then 
\[
e^{-i\sqrt a\xi /2}e^{-\xi^2/24}e^{-i\xi v/2\sqrt{a}}e^{-\xi^2v/24a}=e^{-q'\xi^2}e^{-i\alpha q'\xi}.
\]
Therefore, by Corollary \ref{cor2}
\begin{align}
&\int_{-\infty}^\infty e^{-i\sqrt a\xi /2}e^{-\xi^2/24}e^{-i\xi v/2\sqrt{a}}e^{-\xi^2v/24a}e^{i\xi t}\, d\xi=\int_{-\infty}^\infty e^{-q'\xi^2}e^{-i\alpha q'\xi}e^{i\xi t}\, d\xi\nonumber\\
&\qquad\qquad =\sqrt{\frac{2\pi}{q'}}e^{-q'\alpha^2/4}e^{-\alpha t/2}e^{-t^2/4q}\nonumber\\
&\qquad\qquad =\frac{2\sqrt 3\sqrt{\pi a}}{|q|\sqrt{a+|q|}}(a-v+|q|)e^{-3(a+v)/2}e^{6\sqrt at}e^{-6at^2(a-v)/(a^2+|v|^2)}.
\end{align}
Lastly, we need a result from \cite{UAE}.
\begin{lemma}\label{lem8}
Let $\xi\in\R$ and $a\geq 2$. Then
\be\label{eq7.8}
\left(\sinc \frac{\pi\xi}{\sqrt{a}}\right)^a \leq e^{-\xi^2} + (1 - \chi_{[-1,1]})(\xi/2)\,\frac{2}{(\pi \xi)^2}.
\ee
\end{lemma}
\noindent
Note that the right-hand side of \eqref{eq7.8} is an element of $L^p(\R)$, $1\leq p\leq \infty$, and is independent of $a$.

\begin{theorem}
Suppose $q = a + v\in \bH_{\mathbb R}$ with $a \geq 2$. In the Fourier domain, the quaternionic B-splines $B_q$ converge in $L^p$-norm, $1\leq p\leq \infty$, to a modulated Gaussian:
\[
\left\|\wh{B_q} \left(\frac{\mydot}{\sqrt{a}}\right) - e^{-i\sqrt a(\mydot)/2}e^{-(\mydot)^2/24}e^{-i(\mydot) v/2\sqrt{a}}e^{-(\mydot)^2v/24a}\right\|_p \to 0,
\]
as $a\to\infty$.

In the time domain, the quaternionic B-splines converge in $L_{p^\prime}$-norm, $2\leq p^\prime \leq \infty$, to a modulated Gaussian:
\[
\left\|\sqrt{a}\,{B}_q \left(\sqrt{a} (\mydot)\right) - \frac{2\sqrt 3\sqrt{\pi a}}{|q|\sqrt{a+|q|}}(a-v+|q|)e^{-3(a+v)/2}e^{6\sqrt a (\mydot)}e^{-6a (\mydot)^2(a-v)/(a^2+|v|^2)}\right\|_{p^\prime} \to 0,
\]
as $a\to\infty$.
\end{theorem}

\begin{proof}
Note that $\Xi (\xi) = \frac{1-e^{-i \xi}}{i \xi} = e^{- i \xi/2}\,\sinc (\xi/2)$ and $\wh{B_q} (\xi/\sqrt{a}) = \Xi (\xi/\sqrt{a})^q = \Xi (\xi/\sqrt{a})^a\,e^{v \log \Xi(\xi/\sqrt{a})}$. Moreover, $|e^{v \log \Xi(\xi/\sqrt{a})}| \leq \sqrt{\cosh (\pi |v|)}$. Thus,
\begin{align*}
\left|\wh{B_q} (\xi/\sqrt{a})\right| &\leq  \sqrt{\cosh (\pi |v|)}\, \left|\sinc (\xi/2\sqrt{a})\right|^a\\
&\leq  \sqrt{\cosh (\pi |v|)}\,\left(e^{-(\xi/2\pi)^2}+(1 - \chi_{[-1,1]})(\xi/4\pi)\,\frac{2}{(\xi/2)^2}\right),
\end{align*}
where the bound is independnt of $a$.
Denote by $A_q (\xi/\sqrt{a})$ the approximant of $\wh{B_q}$ given by
\begin{align*}
A_q (\xi/\sqrt{a}) &:= \left(e^{-i\sqrt a\xi/2-\xi^2/24a}\right)^q = \left(e^{-i\sqrt a\xi/2-\xi^2/24}\right)^a\,e^{v (-i\sqrt a\xi/2-\xi^2/24)}\\
& = e^{-i\sqrt a \xi/2}e^{-\xi^2/24}e^{-i\xi v/2\sqrt{a}}e^{-\xi^2v/24a}.
\end{align*}
Then
\begin{align*}
\left|e^{-i\xi v/2\sqrt{a}}\right| &= \left|\cosh(|v|\xi/2\sqrt{a}) - \frac{i v}{|v|}\,\sinh(|v|\xi/2\sqrt{a})\right| = \sqrt{\cosh(|v|\xi/\sqrt{a})}\\
& \leq \sqrt{\cosh(|v|\xi/\sqrt{2})}
\end{align*}
and
\begin{align*}
\left|\left(e^{-i\sqrt a\xi/2-\xi^2/24a}\right)^q\right| &= \left|e^{-i\sqrt a \xi/2}e^{-\xi^2/24}e^{-i\xi v/2\sqrt{a}}e^{-\xi^2v/24a}\right|\\
&\leq \sqrt{\cosh(|v|\xi/\sqrt{2})}\,e^{-\xi^2/24} \leq e^{|v| |\xi|/\sqrt{2}}\,e^{-\xi^2/24}\\
& = e^{3 |v|^2}\,e^{-(|\xi|/\sqrt{24}-\sqrt{3} |v|)^2}.
\end{align*}
Both estimates above hold independent of $a\geq 2$.
We write, as in the proof of Theorem \ref{pointwise cvgence},
\[
\wh{B_q} (\xi/\sqrt{a}) = A_q (\xi/\sqrt{a})\, e^{R(\xi/\sqrt{a})},
\]
where $R(\xi/\sqrt{a}) := q \log \Xi (\xi/\sqrt{a}) - q \left(-i \xi/(2\sqrt{a}) - \xi^2/24a\right)$. The above estimates for $\wh{B_q}$ and $A_q$ imply that both functions are elements of $L^p(\R)$ for $1\leq p \leq \infty$. Therefore, using the fact that there exists a constant $K> 0$ such that
\[
e^{R(\xi/\sqrt{a})} \leq 1 + K\left(\frac{\xi^4}{a^2}\right),\quad\xi\in \R,\;a \geq 2,
\]
we obtain the following estimate: 
\begin{align*}
\left\|\wh{B_q} \left(\frac{\mydot}{\sqrt{a}}\right) - A_q \left(\frac{\mydot}{\sqrt{a}}\right)\right\|_p^p & = \int_\R \left|\wh{B_q} \left(\frac{\xi}{\sqrt{a}}\right) - A_q \left(\frac{\xi}{\sqrt{a}}\right)\right|^p d\xi\\
& = \int_\R \left|A_q \left(\frac{\xi}{\sqrt{a}}\right)\right|^p \left|e^{R(\xi/\sqrt{a})}-1\right|^p d\xi\\
& \leq K^p\, \int_\R \left|A_q \left(\frac{\xi}{\sqrt{a}}\right)\right|^p \left|\frac{\xi^4}{a^2}\right|^p d\xi\\
& \leq \frac{K^p}{a^{2p}} \int_\R \left|A_q \left(\frac{\xi}{\sqrt{a}}\right)\right|^p \left|\xi\right|^{4p} d\xi \to 0\text{ as }a\to\infty .
\end{align*}

To prove the convergence in the time domain, we remark that for $a\geq 1$, $\wh{B_q}\in L^1(\R)\cap L^2(\R)$. The Hausdorff-Young inequality applied to $1\leq p\leq 2$ and $\frac{1}{p^\prime} = 1 - \frac{1}{p}$ and Fourier inversion (see Corollary \ref{cor2} and, in particular, \eqref{eq7.8}) yield
\begin{align*}
&\left\|\sqrt{a}\,{B}_q \left(\sqrt{a} (\mydot)\right) - \frac{2\sqrt 3\sqrt{\pi a}}{|q|\sqrt{a+|q|}}(a-v+|q|)e^{-3(a+v)/2}e^{6\sqrt a (\mydot)}e^{-6a (\mydot)^2(a-v)/(a^2+|v|^2)}\right\|_{p^\prime}
\\
& \hspace*{48pt}\leq c\,\left\|\wh{B_q} \left(\frac{\mydot}{\sqrt{a}}\right) - e^{-i\sqrt a(\mydot)/2}e^{-(\mydot)^2/24}e^{-i(\mydot) v/2\sqrt{a}}e^{(\mydot)^2v/24a}\right\|_p \leq \frac{c'}{a^2},
\end{align*}
for some positive constants $c$ and $c'$.
\end{proof}

\end{document}